\documentclass[11pt]{article}

\usepackage{geometry}
\usepackage{graphicx}

\usepackage[utf8]{inputenc}
\usepackage{pgf}
\usepackage{amsthm}
\usepackage{amsfonts}
\usepackage{amssymb}
\usepackage{mathtools}
\usepackage{tikz-cd}
\usepackage{rotating}
\usepackage{color}
\usepackage{hyperref}
\usepackage{subcaption}

\theoremstyle{plain}

\newtheorem{lemma}{Lemma}[section]
\newtheorem{theorem}[lemma]{Theorem}
\newtheorem{proposition}[lemma]{Proposition}
\newtheorem{corollary}[lemma]{Corollary}
\newtheorem{claim}[lemma]{Claim}
\theoremstyle{definition}
\newtheorem{example}[lemma]{Example}
\newtheorem{remark}[lemma]{Remark}

\newtheorem{question}[lemma]{Question}
\newtheorem{definition}[lemma]{Definition}

\newcommand{\para}[1]{\vspace{2mm}\noindent{\textbf{#1}}}

\newcommand{\B}{\ensuremath{\mathbb{B}}}

\newcommand{\bS}{\ensuremath{\mathbb{S}}}
\newcommand{\Z}{\ensuremath{\mathbb{Z}}}

\newcommand{\vr}[2]{\mathrm{VR}(#1;#2)}
\newcommand{\icech}[2]{\mathrm{\check{C}ech}(#1;#2)}
\newcommand{\cech}[3]{\mathrm{\check{C}ech}(#1,#2;#3)}

\DeclareMathOperator{\nn}{\mathcal{N}}

\DeclareMathOperator{\diam}{diam}
\DeclareMathOperator{\PH}{PH}

\usepackage{todonotes}

\definecolor{darkred}{rgb}{0.8, 0.1, 0.3}

\tikzset{cong/.style={above,edge node={node [sloped, allow upside down, auto=false]{$\cong$}}},
         Isom/.style={above,every to/.append style={edge node={node [sloped, allow upside down, auto=false]{$\cong$}}}}}

\tikzset{simeq/.style={above,every to/.append style={edge node={node [sloped, allow upside down, auto=false]{$\simeq$}}}}}

\tikzset{commutative diagrams/.cd,
mysymbol/.style={start anchor=center,end anchor=center,draw=none}
}

\newcommand{\maxvalidset}       {{maximal valid set}}
\title{On Homotopy Types of Vietoris--Rips Complexes of Metric Gluings}

\author{Micha{\l} Adamaszek, Henry Adams, Ellen Gasparovic, Maria Gommel, Emilie Purvine, \\
Radmila Sazdanovic, Bei Wang, Yusu Wang, and Lori Ziegelmeier}
\date{}

\begin{document}

\maketitle

\begin{abstract}
We study Vietoris--Rips complexes of metric wedge sums and metric gluings.
We show that the Vietoris--Rips complex of a wedge sum, equipped with a natural metric, is homotopy equivalent to the wedge sum of the Vietoris--Rips complexes.
We also provide generalizations for when two metric spaces are glued together along a common isometric subset.
As our main example, we deduce the homotopy type of the Vietoris--Rips complex of two metric graphs 
glued together along a sufficiently short path (compared to lengths of certain loops in the input graphs).
As a result, we can describe the persistent homology, in all homological dimensions, of the Vietoris--Rips complexes of a wide class of metric graphs.
 \end{abstract}


\section{Introduction}
\label{sec:introduction}

When equipped with a notion of similarity or distance, data can be thought of as living in a metric space.
Our goal is to characterize the homotopy types of geometric thickenings of a wide class of metric spaces.
In particular, we consider metric spaces formed by gluing smaller metric spaces together in an admissible fashion, as we will specify.
We then use our results to characterize the persistent homology of these spaces.
Persistent homology is a central tool in topological data analysis that captures complex interactions within a system at multiple scales ~\cite{Carlsson2009,EdelsbrunnerHarer2010}.

The geometric complexes of interest are Vietoris--Rips complexes, which build a simplicial complex on top of a metric space
with respect to a scale parameter $r$.
We first study Vietoris--Rips complexes of metric wedge sums: given two metric spaces $X$ and $Y$ with specified basepoints, the metric wedge sum $X\vee Y$ is obtained by gluing $X$ and $Y$ together at the specified points, and then extending the metrics.
We show that the Vietoris--Rips complex of the metric wedge sum is homotopy equivalent to the wedge sum of the Vietoris--Rips complexes.
We also provide generalizations for certain more general metric gluings, namely, when two metric spaces are glued together along a common isometric subset.

One common metric space that appears in applications such as road networks~\cite{AanChaChe2012}, brain functional networks~\cite{BiswalYetkinHaughton1995}, and the cosmic web~\cite{SousbiePichonKawahara2011} is a \emph{metric graph}, a structure where any two points of the graph (not only vertices) are assigned a distance equal to the minimum length of a path from one point to the other.
In this way, a metric graph encodes the proximity data of a network into the structure of a metric space.
As a special case of our results, we show that the Vietoris--Rips complex of two metric graphs glued together along a sufficiently short common path is homotopy equivalent to the union of the Vietoris--Rips complexes.
This enables us to determine the homotopy types of geometric thickenings of a large class of metric graphs, namely those that can be constructed iteratively via simple gluings.

The motivation for using Vietoris--Rips complexes in the context of data analysis is that under some assumptions, these complexes can recover topological features of an unknown sample space underlying the data.
Indeed, in~\cite{Hausmann1995,Latschev2001}, it is shown that if the underlying space $M$ is a closed Riemannian manifold, if scale parameter $r$ is sufficiently small compared to the injectivity radius of $M$, and if a sample $X$ is sufficiently close to $M$ in the Gromov--Hausdorff distance, then the Vietoris--Rips complex of the sample $X$ at scale $r$ is homotopy equivalent to $M$.
In this paper, we identify the homotopy types of Vietoris--Rips complexes of certain metric graphs at all scale parameters $r$, not just at sufficiently small scales.

Our paper builds on the authors' prior work characterizing the 1-dimensional intrinsic \v{C}ech and Vietoris--Rips persistence modules associated to an arbitrary metric graph.
Indeed,~\cite{GasparovicGommelPurvine2018} shows that the 1-dimensional intrinsic \v{C}ech persistence diagram associated to a metric graph of genus $g$ (i.e., the rank of the 1-dimensional homology of the graph) consists of the points $\left\{ \left(0, \frac{\ell_i}{4}\right) : 1\leq i \leq g\right\}$, where $\ell_i$ corresponds to the length of the $i^{th}$ loop.
In the case of the Vietoris--Rips complex, the results hold with the minor change that the persistence points are $\left\{ \left(0, \frac{\ell_i}{6}\right)~|~1\leq i \leq g\right\}$.
An extension of this work is~\cite{Virk2017}, which studies the 1-dimensional persistence of geodesic spaces.
In~\cite{AdamaszekAdams2017,AdamaszekAdamsFrick2016}, the authors show that the Vietoris--Rips or \v{C}ech complex of the circle obtains the homotopy types of the circle, the $3$-sphere, the $5$-sphere, \ldots, as the scale $r$ increases, giving the persistent homology in all dimensions of a metric graph consisting of a single cycle.
In this paper, we extend this to a larger class of graphs: our results characterize the persistence profile, in any homological dimension, of Vietoris--Rips complexes of metric graphs that can be iteratively built by gluing trees and cycles together along short paths.

Our results on Vietoris--Rips complexes of metric gluings have implications for future algorithm development along the line of ``topological decompositions."
The computation time of homotopy, homology, and persistent homology depend on the size of the simplicial complex.
It would be interesting to investigate if our Theorem~\ref{thm:graph-gluing} means that one can break a large metric graph into smaller pieces, perform computations on the simplicial complex of each piece, and then subsequently reassemble the results together.
This has the potential to use less time and memory.

\para{Outline.} Section~\ref{sec-background} introduces the necessary background and notation.
Our main results on the Vietoris--Rips complexes of metric wedge sums and metric gluings are established in Section~\ref{sec:homotopy}.
In addition to proving homotopy equivalence in the wedge sum case, we show that the persistence module of the wedge sum of the complexes is isomorphic to the persistence module of the complex for the wedge sum.
We develop the necessary machinery to prove that the Vietoris--Rips complex of metric spaces glued together along a sufficiently short path is homotopy equivalent to the union of the Vietoris--Rips complexes.
In Section~\ref{sec:types}, we describe families of metric graphs to which our results apply and furthermore discuss those that we cannot yet characterize.
In Section~\ref{sec:sup}, we describe Vietoris--Rips complexes of gluings of subsets of product spaces,  equipped with the supremum metric.
We conclude in Section~\ref{sec:discussion} with a discussion of our overarching goal of characterizing the persistent homology profiles of large families of metric graphs.

An extended abstract of the present paper previously appeared as a conference paper~\cite{AdamaszekAdamsGasparovic2018}.
The current paper contains the following extensions beyond~\cite{AdamaszekAdamsGasparovic2018}.
We extend~\cite[Lemma~2]{AdamaszekAdamsGasparovic2018} from joins over a single simplex to joins over a collapsible subcomplex in Lemma~\ref{lem:twoplusplus}.
Lemma~\ref{lem:compactness-Quillen} is a more general extension of~\cite[Lemma~3]{AdamaszekAdamsGasparovic2018} when passing from gluings of finite metric spaces to infinite ones.
We also include the proofs of some corollaries and propositions that were given without proof in the conference version of this paper.
Theorem~\ref{thm:general-gluing-edited} extends~\cite[Theorem~8]{AdamaszekAdamsGasparovic2018}.
As pointed out by Wojciech Chach\'{o}lski, Alvin Jin, Martina Scolamiero, and Francesca Tombari, ~\cite[Corollary~9]{AdamaszekAdamsGasparovic2018} is incorrect as stated.
We describe their counterexample in Appendix~\ref{appendix:counterexample}.
This corollary was previously used in proving Theorem 10 of~\cite{AdamaszekAdamsGasparovic2018}.
However, we now remove the dependence on this incorrect result in Theorem~\ref{thm:graph-gluing}, proving a stronger result than in~\cite[Theorem~10]{AdamaszekAdamsGasparovic2018}.
This has implications for further classes of metric graph gluings as discussed in Section~\ref{sec:types}.
We also include a more thorough discussion of gluings with the \v{C}ech complex.
Finally, Section~\ref{sec:sup}, on gluings of spaces using the supremum metric, is entirely new.


\section{Background}
\label{sec-background}

In this section, we recall the relevant background in the settings of
simplicial complexes and metric spaces, including metric graphs.
For a more comprehensive introduction of related concepts from algebraic topology, we refer the reader to~\cite{Hatcher2002}, and to~\cite{Kozlov2008} and~\cite{EdelsbrunnerHarer2010} for a combinatorial and computational treatment, respectively.

\para{Simplicial complexes.}
An \emph{abstract simplicial complex} $K$ is a collection of finite subsets of some (possibly infinite) vertex set $V=V(K)$, such that if $\sigma\in K$ and $\tau\subseteq\sigma$, then $\tau\in K$.
In this paper, we use the same symbol $K$ to denote both the abstract simplicial complex and its geometric realization.
For $V'\subseteq V$, we let $K[V']$ denote the induced simplicial complex on the vertex subset $V'$.
The \emph{join} of two disjoint simplices $\sigma = \{x_0,
\cdots, x_n\}$ and $\tau = \{y_0, \cdots, y_m\}$ is the simplex $\sigma \cup \tau: = \{x_0, \cdots, x_n, y_0, \cdots, y_m\}$.
If $K$ and $L$ are simplicial complexes with disjoint vertex sets $V(K)$ and $V(L)$, then their \emph{join} $K*L$ is the simplicial complex whose vertex set is $V(K)\cup V(L)$, and whose set of simplices is $K*L=\{\sigma_K\cup\sigma_L~|~\sigma_K\in K\mbox{ and }\sigma_L\in L\}$~\cite[Definition~2.16]{Kozlov2008}.

By an abuse of notation, a simplex $S\in K$ can be considered as either a single simplex, or as a simplicial complex $\{S'~|~S'\subseteq S\}$ with all subsets as faces.
When taking joins, we use $\cup$ to denote that the result is a simplex, and we use $*$ to denote that the result is a simplicial complex.
For example, for $a\in V(K)$ a vertex and $S\in K$ a simplex, we use the notation $a\cup S :=\{a\} \cup S$ to denote the simplex formed by adding vertex $a$ to $S$.
We instead use $a*S:=\{S', a\cup S'~|~S'\subseteq S\}$ to denote the associated simplicial complex.
Similarly, for two simplices $\sigma,S\in K$, we use $\sigma\cup S$ to denote a simplex, and we instead use $\sigma*S:=\{\sigma'\cup S'~|~\sigma'\subseteq\sigma,\ S'\subseteq S\}$ to denote the associated simplicial complex.
We let $\dot{S}$ be the boundary simplicial complex $\dot{S}=\{S'~|~S'\subsetneq S\}$, and therefore $a*\dot{S} : = \{S', a \cup S'~|~S'\subsetneq S\}$ and $\sigma*\dot{S} : = \{\sigma'\cup S'~|~\sigma'\subseteq\sigma,\ S'\subsetneq S\}$ are simplicial complexes.

A simplicial complex $K$ is equipped with the topology of a CW-complex~\cite{Hatcher2002}: a subset of the geometric realization of $K$ is closed if and only if its intersection with each finite-dimensional skeleton is closed.

\para{Simplicial collapse.} Recall that if $\tau$ is a face of a simplex $\sigma$, then $\sigma$ is said to be a \emph{coface} of $\tau$.
Given a simplicial complex $K$ and a maximal simplex $\sigma\in K$, we say that a face $\tau\subsetneq\sigma$ is a  \emph{free face of $\sigma$} if $\sigma$ is the unique maximal coface of $\tau$ in $K$.	
A \emph{simplicial collapse} of $K$ with respect to a pair $(\tau, \sigma)$, where $\tau$ is a free face of $\sigma$, is the removal of all simplices $\rho$ such that $\tau \subseteq \rho \subseteq \sigma$.
If $\dim(\sigma) = \dim(\tau)+1$ then this is an \emph{elementary simplicial collapse}.
If $L$ is obtained from a finite simplicial complex $K$ via a sequence of simplicial collapses, then $L$ is homotopy equivalent to $K$, denoted $L\simeq K$~\cite[Proposition~6.14]{Kozlov2008}.

\para{Metric spaces.}
Let $(X,d)$ be a metric space, where $X$ is a set equipped with a distance function $d$.
Let $B(x, r) : = \{y \in X \mid d(x,y) \leq r\}$ denote the closed ball with center $x \in X$ and radius $r\geq 0$.
The \emph{diameter} of $X$ is $\diam(X)=\sup\{d(x,x')\mid x,x'\in X\}$.
A \emph{submetric space} of $X$ is any set $X' \subseteq X$ with its distance function defined by restricting $d$ to $X' \times X'$.

\para{Vietoris--Rips and \v{C}ech complexes.}
We consider two types of simplicial complexes constructed from a metric space $(X, d)$.
These constructions depend on the choice of a scale parameter $r\ge0$.
First, the \emph{Vietoris--Rips complex} of $X$ at scale $r\ge0$ consists of all finite subsets of diameter at most $r$, that is, $\vr{X}{r} = \{\text{finite }\sigma \subseteq X\mid\diam(\sigma) \leq r\}$.
Second, for $X$ a submetric space of $X'$, we define the \emph{ambient \v{C}ech complex} with vertex set $X$ as $\cech{X}{X'}{r} : = \{ \text{finite }\sigma \subseteq X \mid \exists~x' \in X' \text{ with } d(x, x') \leq r\ \forall~x \in \sigma\}$.
The set $X$ is sometimes called the set of ``landmarks'', and $X'$ is called the set of ``witnesses''~\cite{ChazalSilvaOudot2013}.
This complex can equivalently be defined as the nerve of the balls $B_{X'}(x,r)$ in $X'$ that are centered at points $x\in X$, that is, $\cech{X}{X'}{r} = \{\text{finite }\sigma \subseteq X \mid \bigcap_{x \in \sigma} B_{X'}(x,r) \neq \emptyset \}.$ When $X=X'$, we denote the \emph{(intrinsic) \v{C}ech complex} of $X$ as $\icech{X}{r}=\cech{X}{X}{r}$.
Alternatively, the \v{C}ech complex can be defined with an open ball convention, and the Vietoris--Rips complex can be defined as $\vr{X}{r} = \{\sigma \subseteq X \mid \diam(\sigma) < r\}$.
Unless otherwise stated, all of our results hold for both the open and closed convention for \v{C}ech complexes, as well as for both the $<$ and $\le$ convention on Vietoris--Rips complexes.

\para{Persistent homology.} For $k$ a field, for $i\ge 0$ a homological dimension, and for $Y$ a filtered topological space, we denote the \emph{persistent homology} (or \emph{persistence}) \emph{module} of $Y$ by $\PH_i(Y;k)$.
Persistence modules form a category~\cite[Section~2.3]{ChazalDe-SilvaGlisse2016}, where morphisms are given by commutative diagrams.

\para{Gluings of topological spaces.}
Let $X$ and $Y$ be two topological spaces that share a common subset $A=X\cap Y$.
The \emph{gluing} space $X\cup_A Y$ is formed by gluing $X$ to $Y$ along their common subspace $A$.
More formally, let $\iota_X\colon A\to X$ and $\iota_Y\colon A\to Y$ denote the inclusion maps.
Then the gluing space $X\cup_A Y$ is the quotient space of the disjoint union $X\sqcup Y$ under the identification $\iota_X(a)\sim\iota_Y(a)$ for all $a\in A$.
In particular, for $X$ and $Y$ simplicial complexes and $A$ a common subcomplex, the gluing $X\cup_A Y$ is obtained by identifying common faces and is itself a simplicial complex.

\para{Gluings of metric spaces.}
Following Definition 5.23 in~\cite{BridsonHaefliger1999}, we define a way to glue two metric spaces along a closed subspace.
Let $X$ and $Y$ be arbitrary metric spaces with closed subspaces $A_X \subseteq X$ and $A_Y \subseteq Y$.
Let $A$ be a metric space with isometries $\iota_X : A \rightarrow A_X$ and $\iota_Y : A \rightarrow A_Y$.
Let $X \cup_A Y$ denote the quotient of the disjoint union of $X$ and $Y$ by the equivalence between $A_X$ and $A_Y$, i.e.,
$X\cup_A Y = X\sqcup Y / \{\iota_X(a) \sim \iota_Y(a)~|~a \in A\}$.
Then $X\cup_A Y$ is the \emph{gluing of $X$ and $Y$ along $A$.}
We define a metric on $X\cup_A Y$, which
extends the metrics on $X$ and $Y$:
\[ d_{X \cup_A Y}(s,t) = \begin{cases}
d_X(s,t) &\mbox{if }s,t \in X \\
d_Y(s,t) &\mbox{if }s,t \in Y \\
\inf_{a \in A} d_X(s,\iota_X(a)) + d_Y(\iota_Y(a),t) &\mbox{if }s \in X, t \in Y.
\end{cases} \]
Lemma~5.24 of~\cite{BridsonHaefliger1999} shows that the gluing $(X \cup_A Y, d_{X\cup_A Y})$ of two metric spaces along common isometric closed subsets is itself a metric space.
In this paper all of our metric gluings will be done in the case where $X \cap Y = A$ and the $\iota_X$ and $\iota_Y$ are identity maps.
This definition of gluing metric spaces agrees with that of gluing their respective topological spaces, with the standard metric ball topology.

\para{Pointed metric space and wedge sum.}
A {\em pointed metric space} is a metric space $(X,d_X)$ with a distinguished basepoint $b_X \in X$.
In the notation of metric gluings, given two pointed metric spaces $(X,d_X)$ and $(Y,d_Y)$, let $X \vee Y=X\cup_A Y$ where $A_X=\{b_X\}$ and $A_Y=\{b_Y\}$; we also refer to $X \vee Y$ as the \emph{wedge sum} of $X$ and $Y$.
Then the gluing metric on $X \vee Y$ is
\[ d_{X \vee Y}(s,t) = \begin{cases}
d_X(s,t) &\mbox{if }s,t \in X \\
d_Y(s,t) &\mbox{if }s,t \in Y \\
d_X(s,b_X) + d_Y(b_Y,t) &\mbox{if }s \in X, t \in Y.
\end{cases} \]

\para{Metric graphs.}
A graph $G$ consists of set $V = \{v_i\}$ of vertices and a set $E = \{e_j\}$ of edges connecting the vertices.
A graph $G$ is a \emph{metric graph} if each edge $e_j$ is assigned
a positive finite length $l_j $~\cite{BridsonHaefliger1999,BuragoBuragoIvanov2001,Kuchment2004}.
Under mild hypotheses\footnote{For every vertex, the lengths of the edges incident to that vertex are bounded away from zero~\cite[Section~1.9]{BridsonHaefliger1999}.}, the graph $G$ can be equipped with a natural metric $d_G$: the distance between any two points $x$ and $y$ (not necessarily vertices) in the metric graph is the infimum of the length of all paths between them.

\para{Loops of a metric graph.}
Let $\mathbb{S}^1$ be the circle.
A \emph{loop} (or \emph{cycle}) of a metric graph $G$ is a continuous injective map $c: \mathbb{S}^1 \to G$.
We also use the word \emph{loop} to refer to the image of this map.
Intuitively, elements of the singular 1-dimensional homology of $G$ may be represented by collections of loops in $G$~\cite{Hatcher2002}.
The \emph{length} of a loop is the length of the image of the map $c$.

\section{Homotopy equivalences for metric gluings}
\label{sec:homotopy}

\subsection{Homotopy lemmas for simplicial complexes}
In this section, we present a series of lemmas that will be vital to our study of homotopy equivalences of simplicial complexes.
We begin with a lemma proved by Barmak and Minian~\cite{BarmakMinian2008} regarding a sequence of elementary simplicial collapses between two simplicial complexes (Lemma~\ref{lem:39}).
We then generalize this lemma in order to use it in the case where the simplicial collapses need not be elementary (Lemma~\ref{lem:twoplusplus}).
While these first few lemmas are relevant in the context of finite metric spaces, Lemma~\ref{lem:compactness-Quillen} and Corollary~\ref{cor:compactness} will be useful when passing to arbitrary metric spaces.
This set of lemmas will later allow us to show that a complex built on a gluing is homotopy equivalent to the gluing of the complexes.

\begin{lemma}[Lemma~3.9 from~\cite{BarmakMinian2008}]
\label{lem:39}
Let $L$ be a subcomplex of a finite simplicial complex $K$.
Let $T$ be a set
of simplices of $K$ which are not in $L$, and let $a$ be a vertex of $L$
which is contained in no simplex of $T$, but such that $a\cup S$ is a simplex of
$K$ for every $S \in T$.
Finally, suppose that $K = L \cup\bigcup_{S \in T} \{S,
a\cup S\}$.
Then $K$ is homotopy equivalent
to $L$ via a sequence of elementary simplicial collapses.
\end{lemma}

In~\cite{BarmakMinian2008}, Barmak and Minian observe that there is an elementary simplicial collapse from $K$ to $L$ if there is a simplex $S$ of $K$ and a vertex $a$ of $L$ not in $S$ such that $K = L \cup \{S, a\cup S\}$ and $L \cap (a*S) = a* \dot{S}$, where $\dot{S}$ denotes the boundary of $S$.
Indeed, $S$ is the free face of the elementary simplicial collapse, and the fact that $a\cup S$ is the unique maximal coface of $S$ follows from $L \cap (a*S) = a * \dot{S}$ (which implies the intersection of $L$ with $S$ is the boundary of $S$).
See Figure~\ref{fig:collapse} (left) for an illustration.

\begin{figure}[t!]
\begin{center}
\includegraphics[width=0.7\linewidth]{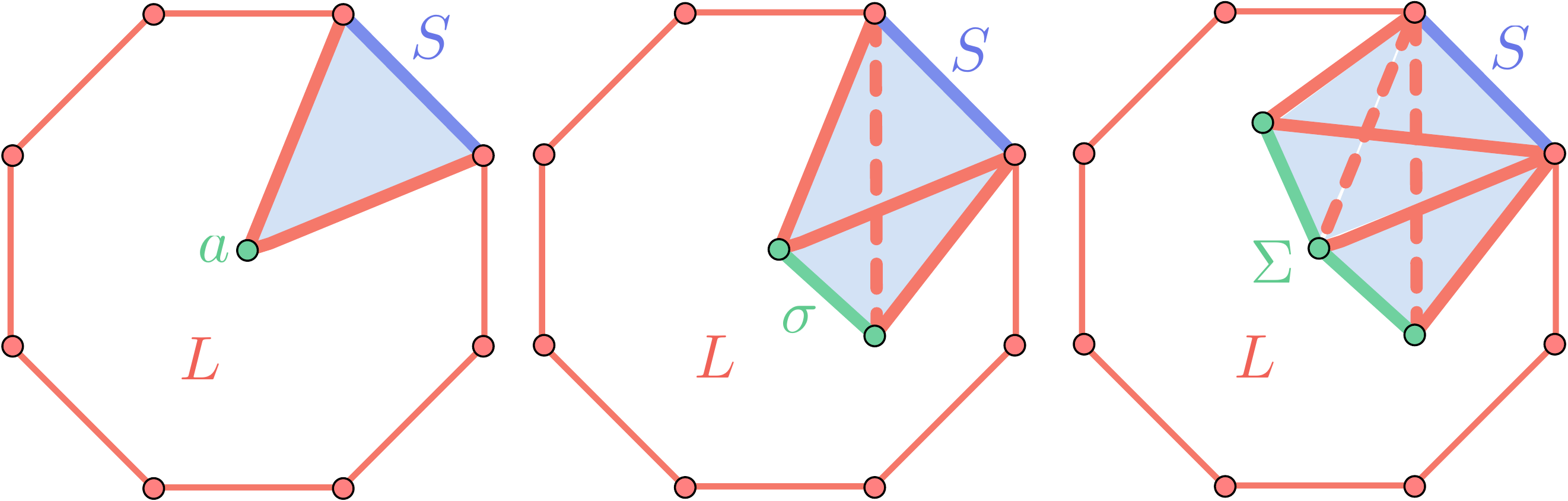}
\caption{From left to right: examples of simplicial collapses when $T = \{S\}$ for Lemma~\ref{lem:39}, Lemma~\ref{lem:gen39}, and Lemma~\ref{lem:twoplusplus}.}
\label{fig:collapse}
\end{center}
\end{figure}

It is not difficult to show that Barmak and Minian's observation can be made more general.
In fact, there is a simplicial collapse from $K$ to $L$ if there is a simplex $S$ of $K$ and another simplex $\sigma$ of $L$, disjoint from $S$, such that $K = L \cup \{\tau : S \subseteq \tau \subseteq \sigma \cup S\}$ and $L \cap (\sigma * S) = \sigma * \dot{S}$.
Indeed, $S$ is again the free face of the simplicial collapse, and the fact that $\sigma \cup S$ is the unique maximal coface of $S$ in $K$ follows from $L \cap (\sigma * S) = \sigma*\dot{S}$ (which implies the intersection of $L$ with $S$ is the boundary of $S$).
See Figure~\ref{fig:collapse} (middle) for an illustration.

\begin{lemma}[Generalization of Lemma~\ref{lem:39}]
\label{lem:gen39}
Let $L$ be a subcomplex of a finite simplicial complex $K$, and let $\sigma$ be a simplex in $L$.
Suppose $T$ is a set of simplices of $K$ which are not in $L$ and which are disjoint from $\sigma$, but such that $\sigma \cup S$ is a simplex of $K$ for every $S\in T$.
Finally, suppose $K = L \cup\bigcup_{\{S \in T\}} \{\tau~|~S\subseteq\tau\subseteq \sigma\cup S\}$.
Then $K$ is homotopy equivalent to $L$ via a sequence of simplicial collapses.
\end{lemma}

\begin{proof}
We mimic the proof of Lemma~3.9 in~\cite{BarmakMinian2008}, except that  we perform a sequence of simplicial collapses rather than elementary simplicial collapses.
Order the elements $S_1$, $S_2$, \ldots, $S_n$ of $T$ in such a way that for every $i$, $j$  with $i\le j$, we have $|S_i|\le|S_j|$.
Define $\displaystyle K_i=L\cup \bigcup {}_{j=1}^i\{\tau~|~S_j\subseteq\tau\subseteq \sigma\cup S_j\}$ for $0\le i\le n$.
Let $S\subsetneq S_i$.
If $S\in T$, then $\sigma\cup S\in K_{i-1}$ since $|S|<|S_i|$.
If $S\notin T$, then $\sigma \cup S$ is in $L \subseteq K_{i-1}$.
This proves that $K_{i-1}\cap (\sigma * S_i) = \sigma * \dot{S_i}$, and so $S_i$ is the free face of a simplicial collapse from $K_i$ to $K_{i-1}$.
Then we are done since $K=K_n$ and $L=K_0$.
\end{proof}

We will use Lemma~\ref{lem:gen39} in the proof of Theorem~\ref{thm:general-gluing}.
One further generalization, Lemma~\ref{lem:twoplusplus}, is needed for the proof of the more general Theorem~\ref{thm:general-gluing-edited}.

\begin{remark}
Let $L$ be a simplicial complex, let $\Sigma$ be a subcomplex of $L$ that is collapsible, and let $u$ be a vertex not in $L$.
We claim that the space $K=L\cup(\Sigma*u)=L\cup\bigcup_{\sigma\in\Sigma}\{\sigma\cup u\}$ is homotopy equivalent to $L$.
Indeed, since $\Sigma$ is collapsible, let
\[
\sigma_0 \searrow \sigma_1,\ \sigma_2 \searrow \sigma_3,\ \ldots,\ \sigma_{2i} \searrow \sigma_{2i+1},\ \ldots,\ \sigma_{2k} \searrow \sigma_{2k+1}
\]
be a sequence of simplicial collapses from $\Sigma$ down to a single vertex $v$.
Here, $\sigma_{2i} \searrow \sigma_{2i+1}$ indicates that
$\sigma_{2i}$ is a free face of $\sigma_{2i+1}$, and that all simplices $\tau$ with $\sigma_{2i}\subseteq \tau \subseteq \sigma_{2i+1}$ are removed from $\Sigma$ via a (possibly non-elementary) simplicial collapse.
Then in $K$ we can perform the sequence of simplicial collapses
\[
\sigma_0\cup\{u\} \searrow \sigma_1\cup\{u\},\ \sigma_2\cup\{u\} \searrow \sigma_3\cup\{u\},\ \ldots,\ \sigma_{2k}\cup\{u\} \searrow \sigma_{2k+1}\cup\{u\},
\ \{u\}\searrow\{v,u\}
\]
giving a simplicial collapse from $K$ down to $L$.
Indeed, $\sigma_{2i}\cup\{u\}$ is a free face of $\sigma_{2i+1}\cup\{u\}$ in the relevant stage of the collapse since $\sigma_{2i}$ is a free face of $\sigma_{2i+1}$ in the collapse of $\Sigma$. 
Furthermore, once these collapses have been made, $\{u\}$ becomes a free face of $\{v,u\}$ (and thus we have the last simplicial collapse $\{u\}\searrow\{v,u\}$ in the sequence above).
Once all simplices containing $u$ have been removed, we have collapsed from $K$ down to $L$.
\end{remark}

We generalize the above remark in the following lemma.
See Figure~\ref{fig:collapse} (right) for an illustration.

\begin{lemma}
\label{lem:twoplusplus}
Let $L$ be a subcomplex of a finite simplicial complex $K$, and suppose $\Sigma$ is a subcomplex of $L$ that is collapsible.
Let $T$ be a set of simplices of $K$ that are not in $L$ and that are disjoint from $\Sigma$, but such that $\sigma \cup S$ is a simplex of $K$ for every $S\in T$ and $\sigma \in \Sigma$.
Finally, suppose $K = L \cup\bigcup_{\sigma\in\Sigma}\bigcup_{S \in T} \{\tau~|~S\subseteq\tau\subseteq \sigma\cup S\}$.
Then $K$ is homotopy equivalent to $L$ via a sequence of simplicial collapses.
\end{lemma}

\begin{proof}
We begin with the case when $T=\{S\}$ consists of a single simplex $S$.
In this case, the condition $K = L \cup\bigcup_{\sigma\in\Sigma} \{\tau~|~S\subseteq\tau\subseteq \sigma\cup S\}$ is equivalent to saying that $\Sigma*\dot{S}$ is a subcomplex of $L$ and $K=L\cup(\Sigma*S)$.
The same argument as in the remark above holds; however, note that we collapse $\Sigma*S\subseteq K$ down to $\Sigma*\dot{S}\subseteq L$, rather than down to $\Sigma$.
In particular, since $\Sigma$ is collapsible, let
\[
\sigma_0 \searrow \sigma_1,\ \sigma_2 \searrow \sigma_3,\ \ldots,\ \sigma_{k-1} \searrow \sigma_k
\]
be a sequence of simplicial collapses from $\Sigma$ down to a single vertex $v$.
Then in $K$ we can perform the sequence of simplicial collapses
\[\sigma_0\cup S \searrow \sigma_1\cup S,\ \sigma_2\cup S \searrow \sigma_3\cup S,\ \ldots,\ \sigma_{k-1}\cup S \searrow \sigma_k\cup S,\ S\searrow \{v\}\cup S\]
yielding a simplicial collapse from $K$ down to $L$.
Indeed, $\sigma_{2i}\cup S$ is a free face of $\sigma_{2i+1}\cup S$ in the relevant stage of the collapse since $\sigma_{2i}$ is a free face of $\sigma_{2i+1}$ in the collapse of $\Sigma$.
Furthermore, once these collapses have been made, $S$ becomes a free face of $\{v\}\cup S$ and thus we can perform simplicial collapse $S\searrow \{v\}\cup S$.
Now that all simplices containing all of $S$ have been removed, we have collapsed from $K$ down to $L$.

We move to the general setting where $T$ may consist of more than a single simplex.
Order the elements $S_1$, $S_2$, \ldots, $S_n$ of $T$ in such a way that for every $i$, $j$  with $i\le j$, we have $|S_i|\le|S_j|$.
Define $K_i=L\cup\bigcup_{\sigma\in\Sigma} \bigcup {}_{j=1}^i\{\tau~|~S_j\subseteq\tau\subseteq \sigma\cup S_j\}$ for $0\le i\le n$.
Let $S\subsetneq S_i$.
If $S\in T$, then $\sigma\cup S\in K_{i-1}$ for all $\sigma\in\Sigma$ since $|S|<|S_i|$.
If $S\notin T$, then $\sigma \cup S$ is in $L \subseteq K_{i-1}$ for all $\sigma\in\Sigma$.
This proves that $K_{i-1}\cap (\Sigma * S_i) = \Sigma * \dot{S_i}$, and so $K_i$ simplicially collapses down onto $K_{i-1}$ by the special case in the prior paragraph.
This completes the proof since $K=K_n$ and $L=K_0$.
\end{proof}

The next lemma will be useful when passing from gluings of finite metric spaces to gluings of arbitrary metric spaces.
Let $\pi_0(Y)$ denote the set of path-connected components of a topological space $Y$.
\begin{lemma}
\label{lem:compactness-Quillen}
Let $g\colon L\to K$ be a simplicial map between (possibly infinite) simplicial complexes.
Let the vertex set of $L$ be $V$, and let the vertex set of $K$ be $W$.
Suppose that for every finite $V_0\subseteq V$ and finite $W_0\subseteq W$, there exists a finite subset $V_1$ with $V_0\subseteq V_1\subseteq V$ and $W_0\subseteq g(V_1)\subseteq W$ such that the induced map $g_1\colon L[V_1]\to K[g(V_1)]$ is a homotopy equivalence.
Then $g\colon L\to K$ is a homotopy equivalence.
\end{lemma}

\begin{proof}
We will use a compactness argument to show that the induced mapping on homotopy groups $g_* \colon \pi_k(L,b) \to \pi_k(K,g(b))$ is an isomorphism for all $k\ge 0$ and for any basepoint $b$ in the geometric realization of $L$.
The conclusion then follows from Whitehead's theorem~\cite[Theorem~4.5]{Hatcher2002}.

First, suppose we have a based map $f \colon \bS^k \to K$ where $\bS^k$ is the $k$-dimensional sphere.
Since $f$ is continuous and $\bS^k$ is compact, it follows that $f(\bS^k)$ is compact in $K$.
Then by~\cite[Proposition~A.1]{Hatcher2002} we know that $f(\bS^k)$ is contained in a finite subcomplex of $K$.
Therefore, there exists a finite subset $W_0\subseteq W$ so that $f$ factors through $K[W_0] \subseteq K$.
By assumption, there exists a finite subset $V_1$ with $V_1\subseteq V$ and $W_0\subseteq g(V_1)\subseteq W$ such that the map $g_1\colon L[V_1]\to K[g(V_1)]$ is a homotopy equivalence.
Thus, we can find a based map $\widetilde{f} \colon \bS^k \to L[V_1]$ such that $[g_1 \widetilde{f}]=[f] \in \pi_k(K[g(V_1)],g(b))$ and hence $[g \widetilde{f}]=[f] \in \pi_k(K,g(b))$.
This proves that $g_*$ is surjective.

Next, suppose that $f \colon \bS^k \to L$ is a based map such that $g f \colon \bS^k \to K$ is null-homotopic.
Let $F \colon \B^{k+1} \to K$ be a null-homotopy between $g f$ and the constant map, where $\B^{k+1}$ is the $(k+1)$-dimensional ball.
By compactness of $\bS^k$ and $\B^{k+1}$, we can find finite subsets $V_0\subseteq V$ and $W_0\subseteq W$ such that $f$ factors through $L[V_0]$ and $F$ factors through $K[W_0]$.
By assumption, there exists a finite subset $V_1$ with $V_0\subseteq V_1\subseteq V$ and $W_0\subseteq g(V_1)\subseteq W$ such that the map $g_1\colon L[V_1]\to K[g(V_1)]$ is a homotopy equivalence.
Note that $g_1 f \colon \bS^k \to K[V_1]$ is null-homotopic via $F$, and since the map $g_1$ is a homotopy equivalence, it follows that $f$ is null-homotopic, and thus $g_*$ is injective.
\end{proof}

By taking $g\colon L\to K$ to be the inclusion of a subcomplex $L$ into a simplicial complex $K$ on the same vertex set, we obtain the following corollary.
This corollary will be the version we most frequently use when passing from wedge sums or gluings of finite metric spaces to wedge sums or gluings of arbitrary metric spaces.
However, Lemma~\ref{lem:compactness-Quillen} in its full generality is needed in Section~\ref{sec:sup}.

\begin{corollary}
\label{cor:compactness}
Let $K$ be a (possibly infinite) simplicial complex with vertex set $V$, and let $L$ be a subcomplex also with vertex set $V$.
Suppose that for every finite $V_0\subseteq V$, there exists a finite subset $V_1$ with $V_0\subseteq V_1\subseteq V$ such that the inclusion $L[V_1]\hookrightarrow K[V_1]$ is a homotopy equivalence.
Then the inclusion map $\iota\colon L\hookrightarrow K$ is a homotopy equivalence.
\end{corollary}

\begin{proof}
The hypothesis in Lemma~\ref{lem:compactness-Quillen} about the existence of the subset $V_1$ follows from the analogous hypothesis in Corollary~\ref{cor:compactness}.
Indeed, choose $V_0$ in Corollary~\ref{cor:compactness} to be the union of $V_0$ and $W_0$, which is in $V$ since $K$ and $L$ have the same vertex sets, in order to obtain a set $V_1$ satisfying the hypothesis of Lemma~\ref{lem:compactness-Quillen}.
\end{proof}

\subsection{Vietoris--Rips and \v{C}ech complexes of wedge sums}\label{sec:wedge}

As a warm-up, we first show in this subsection that the Vietoris--Rips complex of a metric wedge sum (i.e, gluing along a single point) is homotopy equivalent to the wedge sum of the Vietoris--Rips complexes.
In the next subsection, Proposition~\ref{prop:rips_wedge} will be extended in Theorem~\ref{thm:graph-gluing} to gluings of metric graphs along short paths.
Intuitively, such results allow us to characterize the topology of a bigger space via the topology of smaller individual pieces.

Given pointed metric spaces $X$ and $Y$, we use the symbol $b\in X\vee Y$ to denote the point corresponding to the identified distinguished basepoints $b_X\in X$ and $b_Y\in Y$.

\begin{proposition}\label{prop:rips_wedge}
For $X$ and $Y$ pointed metric spaces and $r>0$, we have the homotopy
equivalence
$ \vr{X}{r} \vee \vr{Y}{r} \simeq \vr{X\vee Y}{r}.$
\end{proposition}

\begin{proof}
We first consider the case where $X$ and $Y$ are finite.
We apply Lemma~\ref{lem:39} with $L = \vr{X}{r} \vee \vr{Y}{r}$, with $K=\vr{X\vee Y}{r}$, with $T = \{ \sigma \in K \setminus L~|~b\notin \sigma\}$, and with basepoint $b\in X\vee Y$ serving the role as $a$.
It is easy to check the conditions on $K, L$, and $T$ required by Lemma~\ref{lem:39} are satisfied.
Furthermore, if $\sigma \in T$, then at least one vertex of $X \setminus \{b_X\}$ and one vertex of $Y \setminus \{b_Y\}$ are in $\sigma$.
Hence $\diam(b \cup \sigma) \leq r$ and $b \cup \sigma$ is a simplex of $K$.
Since $K = L \cup \bigcup_{\sigma \in T}\{\sigma, b \cup \sigma \}$, Lemma~\ref{lem:39} implies $L \simeq K$.

Now let $X$ and $Y$ be arbitrary (possibly infinite) pointed metric spaces.
For finite subsets $X_0\subseteq X$ and $Y_0\subseteq Y$ with $b_X\in X_0$ and $b_Y\in Y_0$, the finite case guarantees that $\vr{X_0}{r}\vee\vr{Y_0}{r}\simeq\vr{X_0\vee Y_0}{r}$.
Therefore, we can apply Corollary~\ref{cor:compactness} with $L=\vr{X}{r}\vee\vr{Y}{r}$ and $K=\vr{X\vee Y}{r}$.
\end{proof}

Proposition~\ref{prop:rips_wedge}, in the case of finite metric spaces, is also obtained in~\cite{LesnickRabadanRosenbloom}.

\begin{corollary}\label{cor:rips_wedge}
Let $X$ and $Y$ be pointed metric spaces.
For any homological dimension $i\ge0$ and field $k$, the persistence modules $\PH_i(\vr{X}{r} \vee \vr{Y}{r};k)$ and $\PH_i(\vr{X\vee Y}{r};k)$ are isomorphic.
\end{corollary}
\begin{proof}
For any $r<r'$ we have the following commutative diagram, where all maps are inclusions and where the vertical maps are homotopy equivalences by Proposition~\ref{prop:rips_wedge}:
\begin{center}
\begin{tikzcd}
\vr{X}{r}\vee \vr{Y}{r} \arrow[hookrightarrow]{r} \arrow[hookrightarrow,simeq]{d} & \vr{X}{r'}\vee \vr{Y}{r'} \arrow[hookrightarrow,simeq]{d} \\
\vr{X\vee Y}{r} \arrow[hookrightarrow]{r}& \vr{X\vee Y}{r'}
\end{tikzcd}
\end{center}
Applying homology gives a commutative diagram of homology groups, where the vertical maps are isomorphisms:
\begin{center}
\begin{tikzcd}
H_i(\vr{X}{r}\vee \vr{Y}{r}) \arrow{r} \arrow[Isom]{d} & H_i(\vr{X}{r'}\vee \vr{Y}{r'}) \arrow[Isom]{d} \\%
H_i(\vr{X\vee Y}{r}) \arrow{r} & H_i(\vr{X\vee Y}{r'})
\end{tikzcd}
\end{center}
It follows that $\PH_i(\vr{X}{r} \vee \vr{Y}{r};k)$ and $\PH_i(\vr{X\vee Y}{r};k)$ are isomorphic, and therefore have identical persistence diagrams whenever they're defined.\footnote{Not every persistence module with a real-valued filtration parameter has a corresponding persistence diagram, but isomorphic persistence modules have identical persistence diagrams when they're defined.
It follows from~\cite[Proposition~5.1]{ChazalSilvaOudot2013} that if $X$ is a totally bounded metric space, one can define a persistence diagram for $\PH_i(\vr{X}{r};k)$ and for $\PH_i(\cech{X}{X'}{r};k)$ (where $X\subseteq X'$).}
\end{proof}

We conclude this subsection with some remarks on the analogous results for \v{C}ech complexes.
For a submetric space $X \subseteq X'$, let $\cech{X}{X'}{r}$ be the ambient \v{C}ech complex with landmark set $X$ and witness set $X'$.
Note that if $X \subseteq X'$ and $Y \subseteq Y'$ are pointed with $b_X = b_{X'}$ and $b_Y = b_{Y'}$, then $X \vee Y$ is a submetric space of $X' \vee Y'$.

\begin{proposition}\label{prop:cech_wedge}
For $X \subseteq X'$ and $Y \subseteq Y'$ pointed metric spaces and $r>0$, we have the homotopy equivalence
$\cech{X}{X'}{r} \vee \cech{Y}{Y'}{r} \simeq \cech{X\vee Y}{X'\vee Y'}{r}.$
\end{proposition}

The proof proceeds similarly to the proof of Proposition~\ref{prop:rips_wedge}.

\begin{proof}
Let $d$ be the metric on $X' \vee Y'$, and let $b$ denote the common basepoint in $X\vee Y$ and $X'\vee Y'$.
We first consider the case where $X$ and $Y$ are finite.
We apply Lemma~\ref{lem:39} with $L = \cech{X}{X'}{r} \vee \cech{Y}{Y'}{r}$, with $K= \cech{X\vee Y}{X'\vee Y'}{r}$, and with $T = \{ \sigma \in K \setminus L~|~b\notin \sigma\}$.
Suppose $\sigma \in T$, and let $z \displaystyle \in \cap_{v \in \sigma}B(v, r)$.
Since at least one vertex of each of $X \setminus \{b_X\}$ and $Y \setminus \{b_Y\}$ --- say $x$ and $y$, respectively --- are in $\sigma$, it follows that $d(b,z) \leq \max\{d(x,z), d(y,z) \} \leq r.$
Thus, $b \cup \sigma$ is a simplex of $K$.
Since $K = L \cup \bigcup_{\sigma \in T}\{\sigma, b \cup \sigma\}$, Lemma~\ref{lem:39} implies $L \simeq K$.

We now consider the case where $X$ and $Y$ are arbitrary.
Note that for any finite subsets $X_0\subseteq X$ and $Y_0\subseteq Y$ with $b_X\in X_0$ and $b_Y\in Y_0$, the finite case guarantees that $\cech{X_0}{X'}{r}\vee\cech{Y_0}{Y'}{r}\simeq\cech{X_0\vee Y_0}{X'\vee Y'}{r}$.
Hence, we can apply Corollary~\ref{cor:compactness} with $L=\cech{X}{X'}{r}\vee\cech{Y}{Y'}{r}$ and $K=\cech{X\vee Y}{X'\vee Y}{r}$.
\end{proof}

\begin{corollary}\label{cor:cech_wedge}
For pointed metric spaces $X \subseteq X'$ and $Y \subseteq Y'$ as well as for any homological dimension $i\ge0$ and field $k$, the persistence modules $\PH_i(\cech{X}{X'}{r} \vee \cech{Y}{Y'}{r} )$ and $\PH_i(\cech{X\vee Y}{X'\vee Y'}{r})$ are isomorphic.
\end{corollary}
\begin{proof}
The proof is the same as that for Corollary~\ref{cor:rips_wedge}, except using Proposition~\ref{prop:cech_wedge} instead of Proposition~\ref{prop:rips_wedge}.
\end{proof}

\subsection{Vietoris--Rips complexes of set-wise gluings} \label{sec:VRsetgluings}

We now develop the machinery necessary to prove, in Theorem~\ref{thm:graph-gluing}, that the Vietoris--Rips complex of two metric graphs glued together along a sufficiently short path is homotopy equivalent to the union of the Vietoris--Rips complexes.
First, we prove a more general result for arbitrary metric spaces that intersect in a sufficiently small space.

\begin{theorem}
\label{thm:general-gluing}
Let $X$ and $Y$ be metric spaces with $X\cap Y=A$, where $A$ is a closed subspace of $X$ and $Y$,
and let $r>0$.
Consider $X \cup_A Y$, the metric gluing of $X$ and $Y$ along the intersection $A$.
Suppose that given any $\emptyset\neq S_X\subseteq X\setminus A$ and $\emptyset\neq S_Y\subseteq Y\setminus A$ where $\diam(S_X\cup S_Y)\le r$,
then there is a unique maximal nonempty subset $\sigma\subseteq A$ such that $\diam(S_X\cup S_Y\cup \sigma)\le r$.
Then $\vr{X\cup_A Y}{r}\simeq \vr{X}{r}\cup_{\vr{A}{r}}\vr{Y}{r}$.
In particular, if $\vr{A}{r}$ is contractible, then $\vr{X\cup_A Y}{r}\simeq \vr{X}{r}\vee\vr{Y}{r}$.
\end{theorem}

\begin{proof}
We first restrict our attention to the case when $X$ and $Y$ (and hence $A$) are finite.
Let $n=|A|$.
Order the nonempty subsets $\sigma_1, \sigma_2, \ldots, \sigma_{2^n-1}$ of $A$ so that for every $i,j$ with $i\le j$, we have $|\sigma_i|\ge|\sigma_j|$.
For $i=1,2,\ldots,2^n-1$, let $T_i$ be the set of all simplices of the form $S_X\cup S_Y$ such that:
\begin{itemize}
\item $\emptyset\neq S_X\subseteq X\setminus A$ and $\emptyset\neq S_Y\subseteq Y\setminus A$,
\item $\diam(S_X\cup S_Y)\le r$, and
\item $\sigma_i$ is the (unique) maximal nonempty subset of $A$ satisfying $\diam(S_X\cup S_Y\cup \sigma_i)\le r$.
\end{itemize}

Let $L_0=\vr{X}{r}\cup_{\vr{A}{r}}\vr{Y}{r}$.
We apply Lemma~\ref{lem:gen39} repeatedly to obtain
\begin{align*}
L_0\simeq & L_0 \cup \bigcup_{S \in T_1} \{\tau~|~S\subseteq\tau\subseteq\sigma_1\cup S\}=:L_1\\
\simeq & L_1 \cup\bigcup_{S \in T_2} \{\tau~|~S\subseteq\tau\subseteq\sigma_2\cup S\}=:L_2\\
&\vdots\\
\simeq & L_{2^n-3} \cup\bigcup_{S \in T_{2^n-2}} \{\tau~|~S\subseteq\tau\subseteq\sigma_{2^n-2}\cup S\}=:L_{2^n-2}\\
\simeq & L_{2^n-2} \cup\bigcup_{S \in T_{2^n-1}} \{\tau~|~S\subseteq\tau\subseteq\sigma_{2^n-1}\cup S\}=:L_{2^n-1}.
\end{align*}
The fact that each $L_j$ is a simplicial complex follows since if $S_X\cup S_Y\in T_j$ and $\emptyset \neq S'_X\subseteq S_X$ and $\emptyset \neq S'_Y\subseteq S_Y$, then we have $S'_X\cup S'_Y\in T_i$ for some $i\le j$ (meaning $\sigma_j\subseteq \sigma_i$).
For each $j=1, \ldots, 2^n-1$, we set $K=L_{j}$, $L = L_{j-1}$, $T = T_j$, and $\sigma = \sigma_j$ and apply Lemma~\ref{lem:gen39} to get that $L_{j} \simeq L_{j-1}$ (this works even when $T_j=\emptyset$, in which case $L_j=L_{j-1}$).

To complete the proof of the finite case it suffices to show $L_{2^n-1}=\vr{X\cup_A Y}{r}$.
This is because any simplex $\tau \in \vr{X\cup_A Y}{r} \setminus L_0$ is necessarily of the form $\tau = S_X\cup S_Y\cup \rho$, with $\emptyset\neq S_X \subseteq X\setminus A$, with $\emptyset\neq S_Y \subseteq Y \setminus A$, with $\rho\subseteq A$, and with $\diam(S_X\cup S_Y\cup \rho) \le r$.
(Note that $\rho$ could be the empty set.) 
By the assumptions of the theorem, there exists a unique maximal non-empty set $\sigma_j\subseteq A$ such that $\diam(S_X\cup S_Y\cup \sigma_j) \le r$, and hence $\rho \subseteq \sigma_j$.
Therefore $S_X\cup S_Y \in T_j$, and $\tau$ will be added to $L_j$ since $S_X\cup S_Y \subseteq \tau \subseteq S_X\cup S_Y\cup\sigma_j$.
Hence $\tau \in L_{2^n-1}$ and so $L_{2^n-1} = \vr{X\cup_A Y}{r}$.

Now let $X$ and $Y$ be arbitrary metric spaces.
Note that for any finite subsets $X_0\subseteq X$ and $Y_0\subseteq Y$ with $A_0=X_0\cap Y_0\neq \emptyset$, we have $\vr{X_0}{r}\cup_{\vr{A_0}{r}}\vr{Y_0}{r}\simeq\vr{X_0\cup_{A_0} Y_0}{r}$ by the finite case.
Hence we can apply Corollary~\ref{cor:compactness} with $L=\vr{X}{r}\cup_{\vr{A}{r}}\vr{Y}{r}$ and $K=\vr{X\cup_A Y}{r}$ to complete the proof.
\end{proof}

\begin{remark}\label{rem:general-gluing}
We remark that Theorem~\ref{thm:general-gluing} remains true under the setting in which $X$ and $Y$ are not metric spaces.
Indeed, suppose that $X$ and $Y$ are equipped with arbitrary nonnegative symmetric functions $d_X\colon X\times X\to\mathbb{R}$ and $d_Y\colon Y\times Y\to \mathbb{R}$ with $d_X(x,x) = 0$ for all $x \in X$.
That is, perhaps $d_X$ does not satisfy the triangle inequality, or possibly $d_X(x_1,x_2)=0$ for some distinct values $x_1 \neq x_2$, and similarly for $d_Y$.\footnote{One might also call this a pseudosemimetric.}
If $A$ is a subset of $X$ and $Y$ (no ``closure" assumption is needed, as $X\cup_A Y$ need not be a metric space here), then we can use the same formula as in Section~\ref{sec-background} to define the nonnegative symmetric function $d_{X\cup_A Y}\colon (X\cup_A Y)\times (X\cup_A Y)\to\mathbb{R}$.
We can define Vietoris--Rips complexes and diameters as before.
The proof of Theorem~\ref{thm:general-gluing} still goes through unchanged.
\end{remark}

Next, we provide a generalization of Theorem~\ref{thm:general-gluing}.
While this generalized result is not currently used to produce further results in gluing metric graphs, we believe it may be useful to extend Theorem~\ref{thm:graph-gluing} to glue metric graphs beyond a single path.
In order to state the generalization, we need to define the concept of a \emph{\maxvalidset{}}.
Fix a parameter $r > 0$.
Then, given $\emptyset\neq S_X\subseteq X\setminus A$ and $\emptyset\neq S_Y\subseteq Y\setminus A$, the \emph{\maxvalidset{}} with respect to $S_X$ and $S_Y$ is the collection of all nonempty sets $\sigma_i\subseteq A$ satisfying $\diam(S_X\cup S_Y\cup \sigma_i)\le r$.
We denote it by $\Sigma_{S_X, S_Y} := \{\sigma_1, \ldots, \sigma_m\}$.
This implies that $\diam(\sigma_i) \le r$, and hence, $\sigma_i$ is a simplex in $\vr{A}{r}$.
Moreover, if $\sigma \in \Sigma_{S_X, S_Y}$ and $\tau \subseteq \sigma$ then $\tau \in \Sigma_{S_X, S_Y}$, so $\Sigma_{S_X, S_Y}$ is a simplicial complex.
 
\begin{theorem}
\label{thm:general-gluing-edited}
Let $X$ and $Y$ be metric spaces with $X\cap Y=A$, where $A$ is a closed subspace of $X$ and $Y$,
and let $r>0$.
Consider $X \cup_A Y$, the metric gluing of $X$ and $Y$ along the intersection $A$.
Suppose that given any $\emptyset\neq S_X\subseteq X\setminus A$ and $\emptyset\neq S_Y\subseteq Y\setminus A$ where $\diam(S_X\cup S_Y)\le r$, its \maxvalidset{} $\Sigma_{S_X, S_Y}$ is a (non-empty) collapsible simplicial complex.
Then, $\vr{X\cup_A Y}{r}\simeq \vr{X}{r}\cup_{\vr{A}{r}}\vr{Y}{r}$.
In particular, if $\vr{A}{r}$ is contractible, then $\vr{X\cup_A Y}{r}\simeq \vr{X}{r}\vee\vr{Y}{r}$.
\end{theorem}

\begin{proof}
We first restrict our attention to the case when $X$ and $Y$ (and hence $A$) are finite.
Let $\mathcal{S}(A)$ be the set of all simplicial complexes on the base set $A$
with the containment partial order.
We can use this  partial order to construct a total order on the set of simplicial complexes on $A$ by choosing a linear extension of the containment partial order.
Specifically, let $\Sigma_1, \Sigma_2, \ldots, \Sigma_N \in \mathcal{S}(A)$, where $N$ is the total number of simplicial complexes in $\mathcal{S}(A)$,\footnote{See \url{http://oeis.org/A014466}.}
be such a total order listed in inverse order; that is, if $\Sigma_i \subseteq \Sigma_j$, then $i \ge j$.

Using this total order, we consider a sequence of subsets of $X \cup Y$ for which $\Sigma_i$ is the maximal valid set.
For $i=1,2,\ldots,N$, let $T_i$ be the set of all simplices of the form $S_X\cup S_Y$ satisfying the following conditions:
\begin{itemize}
\item $\emptyset\neq S_X\subseteq X\setminus A$ and $\emptyset\neq S_Y\subseteq Y\setminus A$,
\item $\diam(S_X\cup S_Y)\le r$, and
\item $\Sigma_{S_X, S_Y} = \Sigma_i \in \mathcal{S}(A)$.
\end{itemize}
Note that $T_i$ may be the empty set.
Let $L_0=\vr{X}{r}\cup_{\vr{A}{r}}\vr{Y}{r}$.
We apply Lemma~\ref{lem:twoplusplus} 
repeatedly to obtain
\begin{align*}
L_0\simeq & L_0 \cup \bigcup_{\sigma \in \Sigma_1} \bigcup_{S \in T_1} \{\tau~|~S\subseteq\tau\subseteq\sigma\cup S\} =: L_1 \\
\simeq & L_1 \cup \bigcup_{\sigma \in \Sigma_2}\bigcup_{S \in T_2} \{\tau~|~S\subseteq\tau\subseteq\sigma\cup S\}=:L_2\\
&\vdots\\
\simeq & L_{N-2} \cup \bigcup_{\sigma \in \Sigma_{N-1}}\bigcup_{S \in T_{N-1}} \{\tau~|~S\subseteq\tau\subseteq\sigma\cup S\}=:L_{N-1}\\
\simeq & L_{N-1} \cup \bigcup_{\sigma \in \Sigma_N} \bigcup_{S \in T_{N}} \{\tau~|~S\subseteq\tau\subseteq\sigma\cup S\}=:L_{N}.
\end{align*}
The fact that each $L_j$ is a simplicial complex follows from the following claim.
\begin{claim}\label{claim:maxsetrelation}
If $S_X\cup S_Y\in T_j$ and $\emptyset \neq S'_X\subseteq S_X$ and $\emptyset \neq S'_Y\subseteq S_Y$, then $\Sigma_{S_X, S_Y} \subseteq \Sigma_{S'_X, S'_Y}$.
\end{claim}
\begin{proof}
It is easy to see that for any $\sigma \in \Sigma_{S_X, S_Y}$, we have $\diam(S'_X \cup S'_Y \cup \sigma) \le r$.
Thus there is necessarily a maximal subset $\sigma' \in \Sigma_{S'_X, S'_Y}$ with $\sigma \subseteq \sigma'$.
The claim then follows.
\end{proof}

In fact, the above claim implies that if $S_X\cup S_Y\in T_j$ and $\emptyset \neq S'_X\subseteq S_X$ and $\emptyset \neq S'_Y\subseteq S_Y$, then we have $S'_X\cup S'_Y\in T_i$ for some $i\le j$ (as $\Sigma_{S'_X, S'_Y} \supseteq \Sigma_{S_X, S_Y}$).
Therefore, all faces of newly added simplices are necessarily already present in $T_j$.

For each $j=1, \ldots, N$, we set $K=L_{j}$, $L = L_{j-1}$, $T = T_j$, and $\Sigma = \Sigma_j$ and apply Lemma~\ref{lem:twoplusplus} to obtain $L_{j} \simeq L_{j-1}$.
In particular, if $T_j= \emptyset$, we have $L_j=L_{j-1}$ without applying Lemma~\ref{lem:twoplusplus}.
Otherwise if $T_j\neq \emptyset$, then by the condition in the theorem and the construction of $T_j$, it must be that the maximal valid set $\Sigma_j$ is collapsible.
To see that these choices satisfy the conditions for Lemma~\ref{lem:twoplusplus}, note that $T_j \cap L_{j-1} = \emptyset$.
Indeed, for each $S_X \cup S_Y \in T_j$, its \maxvalidset{} $\Sigma_{S_X, S_Y}$ is unique, and then by the above claim, it is added for the first time at iteration $j$.

To complete the proof of the finite case, it suffices to show that $L_{N}=\vr{X\cup_A Y}{r}$.
This is because any simplex $\tau \in \vr{X\cup_A Y}{r} \setminus L_0$ is necessarily of the form $\tau = S_X\cup S_Y\cup \rho$, with $\emptyset\neq S_X \subseteq X\setminus A$, $\emptyset\neq S_Y \subseteq Y \setminus A$, $\rho\subseteq A$, and $\diam(S_X\cup S_Y\cup \rho) \le r$.
(Note that $\rho$ could be the empty set.) 
Hence, there must exist some non-empty $\sigma \in \Sigma_{S_X, S_Y}$ with $\rho \subseteq \sigma$.
Assume $S_X\cup S_Y \in T_j$ (and thus $\Sigma_{S_X, S_Y} = \Sigma_j$).
Then we will add simplex $\tau$ to $L_j$.
Therefore, $\tau \in L_N$ and so $L_N = \vr{X\cup_A Y}{r}$.

Finally, let $X$, $Y$, and $A$ be arbitrary (possibly infinite) metric spaces.
Note that for any finite subsets $X_0\subseteq X$ and $Y_0\subseteq Y$ with $A_0=X_0\cap Y_0\neq \emptyset$, we have $\vr{X_0}{r}\cup_{\vr{A_0}{r}}\vr{Y_0}{r}\simeq\vr{X_0\cup_{A_0} Y_0}{r}$ by the finite case.
Hence, we can apply Corollary~\ref{cor:compactness} with $L=\vr{X}{r}\cup_{\vr{A}{r}}\vr{Y}{r}$ and $K=\vr{X\cup_A Y}{r}$ to complete the proof of Theorem~\ref{thm:general-gluing-edited}.
\end{proof}

We remark that the condition in Theorems \ref{thm:general-gluing} and \ref{thm:general-gluing-edited} that the maximal set $\sigma$ or the maximal valid set $\Sigma_{S_X \cup S_Y}$ be nonempty is very important.
We thank Wojciech Chach\'{o}lski, Alvin Jin, Martina Scolamiero, and Francesca Tombari for pointing out a counterexample to~\cite[Corollary~9]{AdamaszekAdamsGasparovic2018}, and we would like to mention their forthcoming work on homotopical decompositions of Vietoris--Rips complexes~\cite{Chacholski}.
We describe their counterexample in Appendix~\ref{appendix:counterexample}.
The mistaken proof misses this crucial point of verifying that $\sigma$ is nonempty.
In the conference version of this paper,~\cite[Theorem 10]{AdamaszekAdamsGasparovic2018} relied on the incorrect corollary to prove the $r > \alpha$ case.
In this version, we remove that reliance 
and also prove a stronger version of~\cite[Theorem 10]{AdamaszekAdamsGasparovic2018} in Theorem~\ref{thm:graph-gluing}.

\subsection{Vietoris--Rips complexes of gluings of metric graphs}\label{sec:vr-graph}

We next study gluings of metric graphs.
The setup of the following theorem regarding metric graph gluings is illustrated in Figure~\ref{fig:gluing}.
In a graph, the \emph{degree} of a vertex without self-loops is the number of incident edges to that vertex.
A \emph{path graph} (or simply a path) is one with $n$ ordered vertices and with $n-1$ edges connecting pairs of successive vertices.
Any path graph can be parameterized by a closed interval. We call the points with the smallest and the largest values in this parametrization the \emph{endpoints} of the path.
Examples of how this theorem can be used in characterizing the homotopy types of Vietoris--Rips complexes of certain families of metric graphs will be given in Section~\ref{sec:types}.

\begin{figure}[t!]
\begin{center}
\includegraphics[width=0.98\linewidth]{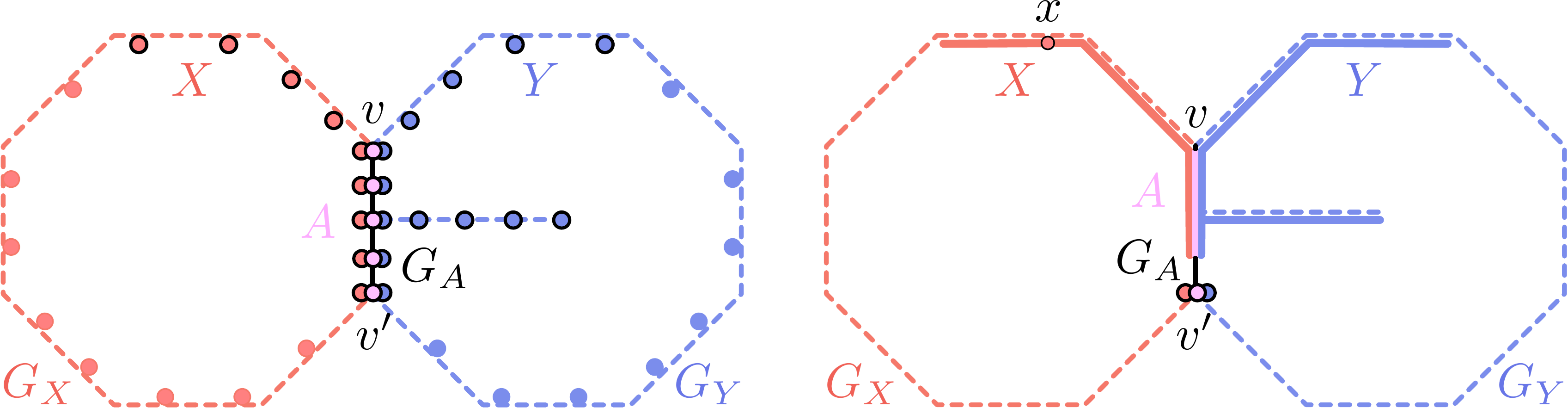}
\caption{Illustration of Theorem~\ref{thm:graph-gluing} on metric graph gluings and both finite (left) and infinite (right) subsets thereof.
The metric graphs $G_X$ and $G_Y$ are shown with thin, dotted red and blue lines respectively; $X$ and $Y$ are shown in the infinite case with thick, solid red and blue lines respectively;
$G_A$ corresponds to the black solid line while $A$ corresponds to the pink  solid line.
The finite subset case uses the same color scheme.
}
\label{fig:gluing}
\end{center}
\end{figure}

\begin{theorem}
\label{thm:graph-gluing}
Let $G=G_X\cup_{G_A} G_Y$ be a metric graph, where $G_A=G_X\cap G_Y$ is a closed metric subgraph of the metric graphs $G_X$ and $G_Y$.
Suppose furthermore that $G_A$ is a path of length $\alpha$, and that every vertex of $G_A$ besides the two endpoints always has degree 2 as a vertex restricted to $G_X$.
Let $\ell_X$ denote the length of the shortest cycle in $G_X$ passing through $G_A$,
and similarly, let $\ell_Y$ denote the length of the shortest cycle intersecting $G_A$ in $G_Y$.
Let $\ell = \min\{\ell_X, \ell_Y\}$.
Suppose $\alpha < \frac{\ell}{3}$.
Let $X\subseteq G_X$ and $Y\subseteq G_Y$ be arbitrary subsets such that $X\cap G_Y = Y\cap G_X =X\cap Y=:A$, where $A$ contains the two endpoints of $G_A$.
Then $\vr{X\cup_A Y}{r}\simeq \vr{X}{r}\cup_{\vr{A}{r}}\vr{Y}{r}$ for all $r>0$.
In particular, if $\vr{A}{r}$ is contractible, then $\vr{X\cup_A Y}{r}\simeq \vr{X}{r}\vee\vr{Y}{r}$.
\end{theorem}

\begin{proof}
We first consider the case where $r<\alpha$.
Let $\emptyset\neq S_X \subseteq X\setminus A$ and $\emptyset\neq S_Y \subseteq Y\setminus A$ be such that $\diam(S_X \cup S_Y) \leq r$.
Let $v$ and $v'$ in $A$ denote the endpoints of $G_A$.
We claim that exactly one of $v$ or $v'$ (without loss of generality let it be $v$) satisfies the property that every path connecting a point $x\in S_X$ to $y\in S_Y$ of length at most $r$ passes through $v$.
First, fix $y\in S_Y$; we claim that every path of length at most $r$ connecting a point in $S_X$ to $y$ must pass through $v$.
Indeed, if not, then pick $x$ and $x'$ whose shortest paths to $y$ go through $v$ and $v'$, respectively.
By connecting these paths with the path from $x$ to $x'$ we find\footnote{The concatenation of these paths may not be injective, but it is not hard to modify the non-injective loop to obtain an (injective) cycle.} a cycle of length at most $3r<3\alpha$.
A portion of this cycle of length at most $\frac{3\alpha}{2}$ lies in either $X$ or $Y$, by adding a subpath from $G_A$ we obtain a cycle in either $X$ or $Y$ of length at most $\frac{5\alpha}{2}<3\alpha$, a contradiction.
In other words, all points in $S_x$ are within distance $r$ to $v$ (as $v$ is along the shortest path from a point $x\in S_x$ to $y$ and $d(x, y) \le r$). 
On the other hand, note that for any point $x\in S_x$, if it is within distance $r$ to $v$, then $d_X(x,v') > r$; as otherwise, there will be a non-trivial cycle in $X$ containing $G_A$ of length $2r + \alpha < 3\alpha$, which contradicts the condition that $\alpha < \frac{\ell}{3}$. 
This then implies that all points in $S_Y$ are within distance $r$ to $v$: Specifically, if there exists $y' \in S_Y$ such that $d(y', v) > r$, then $d(y', x) > r$ for any $x\in S_X$, which contradicts that $diam(S_X \cup S_Y) \le r$. 
Putting everything together, we thus have that every path connecting a point $x\in S_X$ to $y\in S_Y$ of length at most $r$ passes through $v$. 
It is now not hard to see that there is a unique maximal nonempty subset $\tau\subseteq A$ with $\diam(S_X\cup S_Y\cup \tau)\leq r$.
Indeed, $\tau$ is the intersection of $A$ with some subpath of $G_A$ containing $v$. 
Hence, Theorem~\ref{thm:general-gluing} implies that $\vr{X\cup_A Y}{r}\simeq \vr{X}{r}\cup_{\vr{A}{r}}\vr{Y}{r}$, as desired.

We next consider the case where $r \geq \alpha$ and will show that for every $\emptyset\neq S_X \subseteq X \setminus A$ and $\emptyset\neq S_Y \subseteq Y \setminus A$ with $\diam(S_X \cup S_Y) \leq r$, there exists a point $u \in A$ such that $\diam(S_X \cup S_Y \cup \{u\}) \leq r$.
This will ensure that there exists a nonempty $\sigma \subseteq A$ satisfying $\diam(S_X\cup S_Y\cup \sigma)\le r$.
The set of all $\sigma\subseteq A$ satisfying $\diam(S_X\cup S_Y\cup \sigma)\leq r$ is closed under unions since $\diam(A) \leq r$, and hence there will be a unique non-empty maximal $\sigma$.
The conclusion will again follow from Theorem~\ref{thm:general-gluing}.

For any set $Z \subseteq X \cup Y$, we say that $a \in A$ is a \emph{witness} for $Z$ if $\diam(Z \cup \{a\}) \leq r$.
If $Z = \{z\}$, then we say that the point $z$ is \emph{witnessed by} $a$.
Once again, let $v$ and $v'$ denote the endpoints of the path $G_A$. 
Given an $S_X$ and $S_Y$ as above, we first claim generally that for any $z \in S_X \cup S_Y$, the point $z$ is witnessed by either $v$ or $v'$.
Indeed, if $z \in S_X$ then for any $y \in S_Y$ the shortest path from $z$ to $y$, which has length at most $r$, must pass through either $v$ or $v'$ on its way into the gluing region $A$ because of the degree 2 restriction within $G_X$ for vertices in $G_A$.
Therefore the distance from $z$ to one of $v$ or $v'$ (or both) must be at most $r$.
A similar argument holds true if $z \in S_Y$.
If either $v$ or $v'$ is a witness for $S_X \cup S_Y$, then we are done.
So suppose neither $v$ nor $v'$ is a witness for $S_X \cup S_Y$.
We proceed with a case analysis.
In each case we will assume that some point in $S_X \cup S_Y$ is not witnessed by $v'$ and another point is not witnessed by $v$.
We will either derive a contradiction or find a new witness for $S_X \cup S_Y$.

\paragraph{Case 1:}
There is a $y \in S_Y$ not witnessed by $v'$ and an $x \in S_X$ not witnessed by $v$ (the reasoning for this case also holds when $v$ and $v'$ are swapped).
In this case we know that $d(y, v') > r$.
But because of our prior claim, $y$ must then be witnessed by $v$, so that $d(y, v) \leq r$.
Observe that
\[
 r < d(y,v') \leq d(y,v) + d(v,v') \le d(y,v) + \alpha.
\]
This implies $d(y,v) > r-\alpha$.
For all $z \in S_X$, since $d(y,z) \leq r$ via a path going through $v$, it must be that $d(v, z) < \alpha$.
So $S_X \subseteq B_X(v,\alpha)$, i.e., $S_X$ must be contained within an open ball in $X$ of radius $\alpha < r$ centered at $v$.
This is a contradiction to our assumption that there is an $x \in S_X$ not witnessed by $v$, and hence Case~1 is impossible.

\paragraph{Case 2:} 
There is a $y \in S_Y$ not witnessed by $v'$ and a $y' \in S_Y$ not witnessed by $v$.
By the same argument as above, we have $S_X \subseteq B_X(v,\alpha)$, and by swapping $y'$ for $y$ and $v$ for $v'$ in that argument, we have $S_X \subseteq B_X(v', \alpha)$.
Therefore, $S_X \subseteq B_X(v, \alpha) \cap B_X(v', \alpha)$.
Consider an $x \in S_X$, so $d(x, v) < \alpha$ and $d(x, v') < \alpha$.
Because all $a \in A$ with $a \neq v, v'$ have degree 2 in $X$, and $d(v, v')=\alpha$, the shortest path\footnote{This shortest path is unique since there are no cycles of length at most $2\alpha\le3\alpha< \ell$.} from $x$ to $v'$ cannot go through $A$, which implies that it cannot go through $v$. Similarly, the shortest path from $x$ to $v$ cannot go through $v'$.
Hence as pictured in Figure \ref{fig:case3}, let $\hat{x}$ be the point farthest from $x$ which is shared between the shortest path from $x$ to $v$ and the shortest path from $x$ to $v'$.
This creates a non-trivial cycle in $G_X$ of length at most $3\alpha < \ell$, which is a contradiction to the shortest cycle in $X$ having length greater than or equal to $\ell$.
Therefore, all of $S_Y$ must have a common witness.

\paragraph{Case 3:} 
There is an $x \in S_X$ not witnessed by $v'$ and an $x' \in S_X$ not witnessed by $v$.
By a similar argument as above, we have $S_Y \subseteq B_Y(v, \alpha) \cap B_Y(v', \alpha)$.
Now, since the degree of vertices in $A$ restricted to $G_Y$ may be greater than 2 we cannot use quite the same argument as in Case 2.
Instead we will begin by proving the following claim.

\begin{claim}
Suppose we are in the setting of Case~3.
For every $y \in S_Y$, there exists a unique entry point 
$u \in A$ such that for all $z \in S_X$, $d(z, y) = d(z, u) + d(u, y)$.
In other words, when traveling on any shortest path from a given $y$ to any point in $S_X$, one must enter $A$ at a unique point $u \in A$.
\end{claim}

\begin{proof}
We will prove this by contradiction.
Assume there is a $y \in S_Y$ and $z, z' \in S_X$ such that $u$ is the first point in $A$ on a shortest path from $y$ to $z$, and $u'$ is the first point in $A$ on a shortest path from $y$ to $z'$ (see Figure \ref{fig:case4} for an illustration).
Any shortest path from $y$ to $z$ must go through either $v$ or $v'$ in order to get out of the gluing path and into $X$.
Therefore, since both $d(y, v)$ and $d(y, v')$ are at most $\alpha$, and $u$ is on the path from $y$ to $v$, it must be that $d(y, u) \leq \alpha$.
A similar argument gives us that $d(y, u') \leq \alpha$.
These two distances plus $d(u, u') \leq d(v, v') \le \alpha$ gives us a nontrivial cycle in $G_Y$ of length at most $3\alpha < \ell$, which is again a contradiction.
Therefore for every $y \in Y$, there is a unique entry point into $A$ common amongst all shortest paths from $y$ to $S_X$.
\end{proof}
\begin{figure}
\centering
\begin{subfigure}[b]{0.45\textwidth}
    \centering
    \includegraphics[height=1.75in]{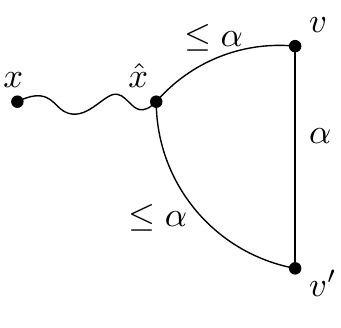}
    \caption{}
    \label{fig:case3}
\end{subfigure}
\qquad
\begin{subfigure}[b]{0.45\textwidth}
    \centering
    \includegraphics[height=1.75in]{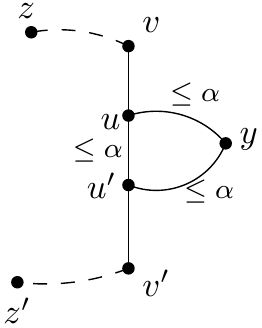}
    \caption{}
    \label{fig:case4}
\end{subfigure}
\caption{Illustrations of the contradictions in (a) Case 2 and (b) Case 3 in proof of Theorem \ref{thm:graph-gluing}.}
\end{figure}

Let $U \subseteq A$ denote the set of all entry points for all $y \in S_Y$.
Then, we claim that any $u \in U$ is a witness for $S_X$.
Indeed, each $u$ is along a shortest path from some $y$ to \emph{all} of $S_X$.
Therefore, for any $x \in S_X$, $d(x, u) \leq d(x, y) \leq r$.

Finally, we will show that any $u\in U$ is also a witness for $S_Y$.
Notice that any $u \in U \subseteq A$ divides the gluing path $G_A$ into two regions which we will call ``above'' $u$ (the shortest path between $u$ and $v$) and ``below'' $u$ (the shortest path between $u$ and $v'$).
Therefore, every other $u' \in U$ is either above $u$ or below $u$.
Recall as proven at the beginning of Case 3 that $S_Y \subseteq B_Y(v, \alpha) \cap B_Y(v', \alpha)$.
In order for the $y$ that matches a given $u$ to get to $v$, the shortest path will go through $u$ and then through all of the $u'$ above $u$.
Therefore, $d(y, u') \leq \alpha \leq r$ for all $u'$ above $u$.
Similarly, for $y$ to get to $v'$, the shortest path will go through $u$ and then through all $u''$ below $u$, giving $d(y, u'') \leq \alpha \leq r$ for all $u''$ below $u$.
This means that for any $y \in S_Y$ and any $u \in U$, $d(y, u) \leq r$ and so $u$ is a witness for $S_Y$.
By the paragraph above any $u\in U$ is a witness also for $S_X$, and hence for $S_X\cup S_Y$.
This completes the proof of Case~3.

This completes the proof for $r \geq \alpha$,
finishing the proof of Theorem~\ref{thm:graph-gluing}.
\end{proof}

\begin{corollary}\label{cor:graph-gluing}
Let $G$, $G_X$, $G_Y$, $G_A$, $X$, $Y$, and $A$ satisfy the same hypotheses as in the statement of Theorem~\ref{thm:graph-gluing}.
Suppose furthermore that $\vr{A}{r}$ is contractible for all $r>0$.
Then for any homological dimension $i\ge0$ and field $k$, the persistence modules $\PH_i(\vr{X}{r}\vee \vr{Y}{r};k)$ and $\PH_i(\vr{X\cup_A Y}{r}; k)$ are isomorphic.
\end{corollary}
\begin{proof}
The proof is the same as that for Corollary~\ref{cor:rips_wedge}, except using Theorem~\ref{thm:graph-gluing} instead of Proposition~\ref{prop:rips_wedge}.
\end{proof}

\subsection{\v{C}ech complexes of set-wise gluings}\label{sec:cech-gluings}

It seems natural to ask if our results in Sections~\ref{sec:VRsetgluings}--\ref{sec:vr-graph} extend to \v{C}ech complexes.
For example, is it necessarily the case that $\icech{X}{r} \cup_{\icech{A}{r}} \icech{Y}{r} \simeq \icech{X \cup_{A} Y}{r}$, where $X,Y$ and $A$ are as described in Theorem~\ref{thm:graph-gluing}?
Interestingly, while the desired result may hold true, the arguments in the proof of Theorem~\ref{thm:graph-gluing} do not all directly transfer to the \v{C}ech case.

Theorem~\ref{thm:general-gluing} can be extended to the \v{C}ech case, with an analogous proof, by replacing the condition $\diam(S_X \cup S_Y \cup \sigma) \le r$ with $\bigcap_{z\in S_X\cup S_Y\cup\sigma}B(z;r)\neq\emptyset$.
Note that if $S_X\cup S_Y\cup\sigma$ is finite, then $\bigcap_{z\in S_X\cup S_Y\cup\sigma}B(z;r)\neq\emptyset$ means that $S_X\cup S_Y\cup\sigma \in \icech{X \cup_{A} Y}{r}$.

\begin{theorem}[\v{C}ech-version of Theorem~\ref{thm:general-gluing}]\label{thm:general-cech-gluing}
Let $X$ and $Y$ be metric spaces with $X\cap Y=A$, where $A$ is a closed subspace of $X$ and $Y$,
and let $r>0$.
Suppose that if $\bigcap_{z\in S_X\cup S_Y}B(z;r)\neq\emptyset$ for some $S_X\subseteq X\setminus A$ and $S_Y\subseteq Y\setminus A$, then there is a unique maximal nonempty subset $\sigma\subseteq A$ such that $\bigcap_{z\in S_X\cup S_Y\cup\sigma}B(z;r)\neq\emptyset$.
Then $\icech{X\cup_A Y}{r}\simeq \icech{X}{r}\cup_{\icech{A}{r}}\icech{Y}{r}$.
Hence if $\icech{A}{r}$ is contractible, we have that $\icech{X\cup_A Y}{r}\simeq \icech{X}{r}\vee\icech{Y}{r}$.
\end{theorem}

\begin{figure}[t!]
\begin{center}
\includegraphics[width=0.525\linewidth]{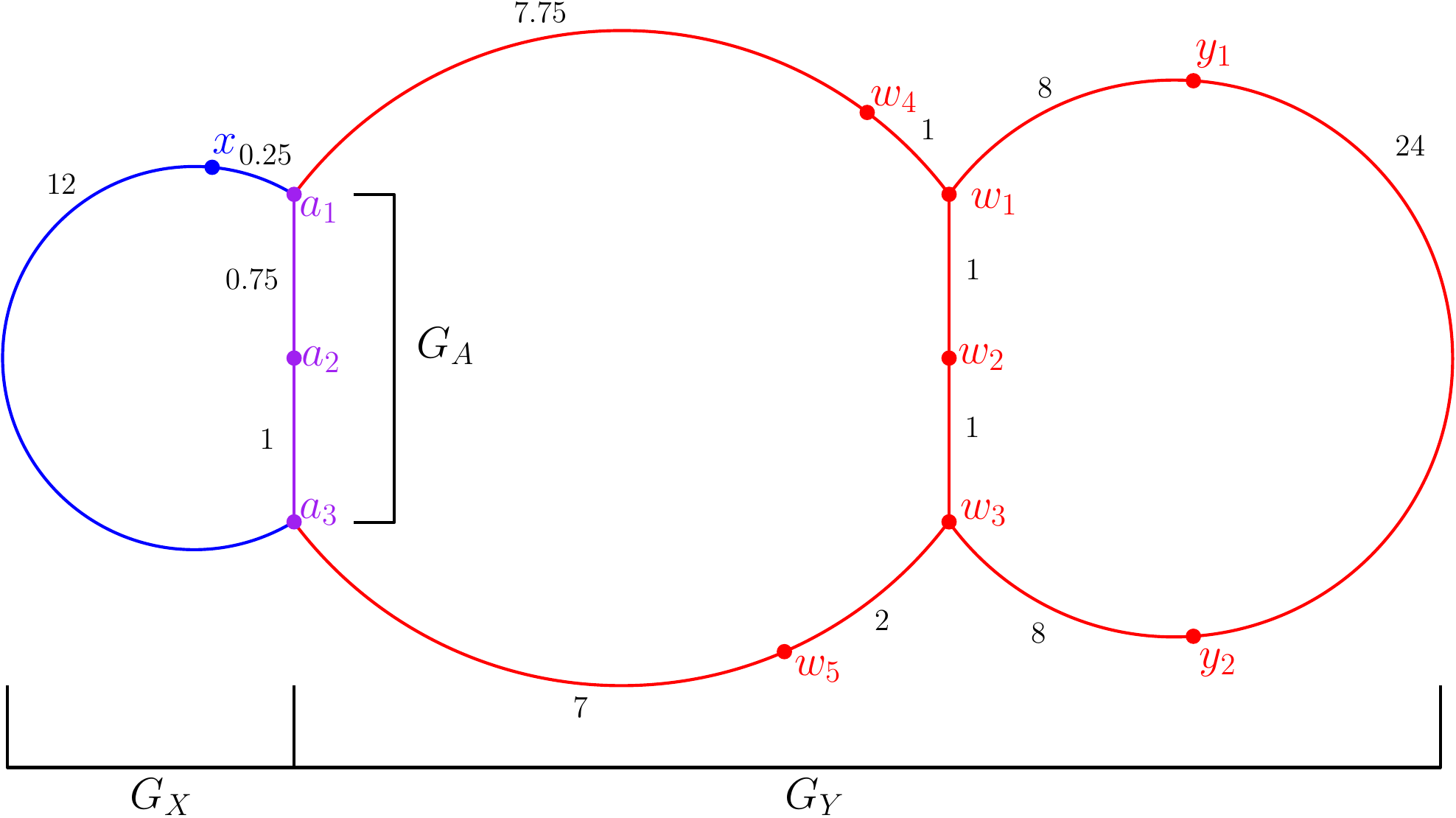}
\caption{Two graphs, $G_X$ (blue) and $G_Y$ (red), glued along $G_A$ (purple).
Distances between points are denoted in black.}
\label{fig:cechexample}
\end{center}
\end{figure}

It is hard to generalize Theorem~\ref{thm:graph-gluing} for the \v{C}ech case with the existing type of simplicial collapses.
We give an example to illustrate this point.
See Figure~\ref{fig:cechexample}, where the distance metric is induced by the lengths marked on the arcs of the graph.
In the notation of Theorem~\ref{thm:graph-gluing}, let $G_X$ be the loop on the left and $G_Y$ be the two loops on the right, where $G_A$ is the path from $a_1$ to $a_3$.
Let $X = \{x,a_1,a_2,a_3\}$ and $Y = \{y_1,y_2,a_1,a_2,a_3,w_1,w_2,w_3, w_4, w_5\}$, so $A = \{a_1,a_2,a_3\}$.
We assume the following distances, as labeled in the figure:  $d(x,a_1) = 0.25,$ $d(a_1,a_2) = 0.75,$ $d(a_2,a_3) = d(w_1,w_2) = d(w_2,w_3) = d(w_1, w_4) = 1$, $d(w_1,y_1) = d(w_3,y_2) = 8$, $d(w_3, w_5) = 2, d(a_1, w_4) = 7.75,$ and  $d(a_3, w_5) = 7$.
The distances between all other pairs of nodes are induced by these lengths.
The scale parameter $r$ for the \v{C}ech complex is $r = 10$.

To prove an analogous version of Theorem~\ref{thm:graph-gluing}, a key step would be to show that when the size of $A$ is small enough, then the following condition in Theorem~\ref{thm:general-cech-gluing} can be satisfied:

\begin{quote}
{\sf (Condition-R):} For any $S_X\subseteq X\setminus A$ and $S_Y\subseteq Y\setminus A$ such that\newline
$\bigcap_{z\in S_X\cup S_Y}B(z;r)\neq\emptyset$, there is a unique maximal nonempty subset $\sigma\subseteq A$ such that $\bigcap_{z\in S_X\cup S_Y\cup\sigma}B(z;r)\neq\emptyset$.
\end{quote}
However, in the example from Figure~\ref{fig:cechexample}, consider $S_X = \{x\}$ and $S_Y = \{y_1, y_2\}$.
It turns out that there are at least two maximal subsets, $\sigma_1 = \{a_1, a_3\}$ and $\sigma_2 = \{a_1, a_2\}$, that form a simplex with $S_X\cup S_Y$ in $\icech{X \cup_{A} Y}{r}$.
(Note that in the finite case, $\bigcap_{z\in S_X\cup S_Y\cup\sigma}B(z;r)\neq\emptyset$ means that $S_X\cup S_Y\cup \sigma$ forms a simplex in $\icech{X \cup_{A} Y}{r}$.)
Indeed, $\bigcap_{z\in S_X\cup S_Y\cup\sigma_1}B(z;r) = \{w_2\}$ and $\bigcap_{z\in S_X\cup S_Y\cup\sigma_2}B(z;r) = \{w_1\}$; however, $\bigcap_{z\in S_X\cup S_Y\cup\sigma_1 \cup \sigma_2}B(z;r) = \emptyset$.
This example can be modified to include all points from the underlying space of this metric graph, but for simplicity we consider the discrete case here.

While this example does not rule out a generalization of Theorem~\ref{thm:graph-gluing} to the \v{C}ech case, it suggests that such a generalization (if it holds) will require a different proof technique, perhaps using Theorem~\ref{thm:general-gluing-edited}.


\section{Applicability to certain families of graphs}
\label{sec:types}

The results in Section~\ref{sec:homotopy} provide a mechanism to compute the homotopy types and persistent homology of Vietoris--Rips complexes of metric spaces built from gluing together simpler ones.
For the sake of brevity, if the results of Section~\ref{sec:homotopy} can be used to completely describe the homotopy types and persistence module of the Vietoris--Rips complexes of metric space $X$, then we will simply say that space $X$ can be \emph{characterized}.
In Figure~\ref{fig:types}, the three metric graphs (a)-(c) can be characterized, whereas (d) cannot.
We first describe some families of metric spaces that can be characterized, and discuss obstructions to characterization, such as in the case of example (d).

\begin{figure}[!h]
\begin{center}
\includegraphics[width=0.9\textwidth]{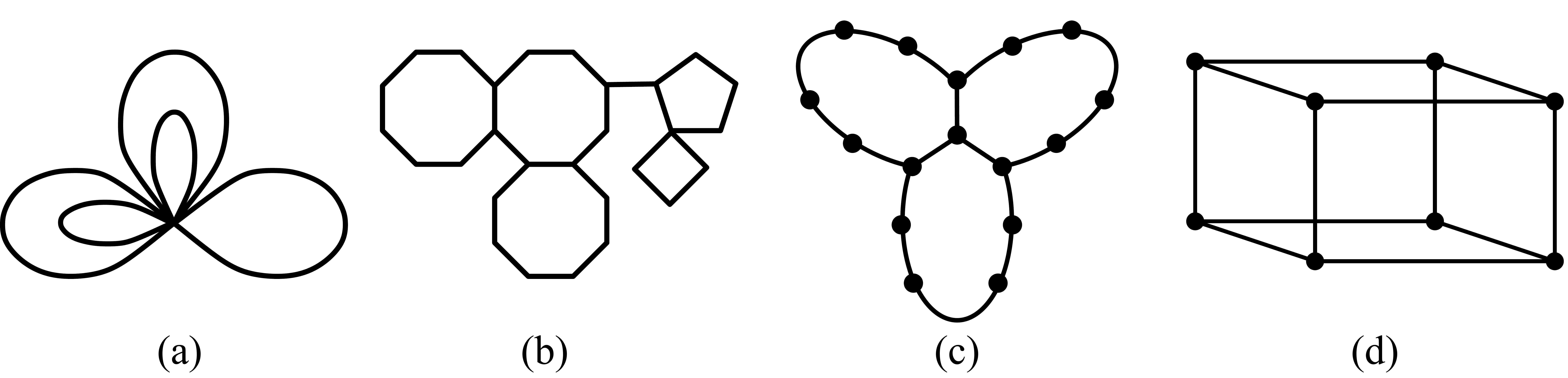}
\caption{Graphs (a), (b), and (c) can be characterized while (d) cannot.}
\label{fig:types}
\end{center}
\end{figure}

We consider finite metric spaces and metric graphs that can be understood using the results in this paper.
Examples of finite metric spaces whose Vietoris--Rips complexes are well-understood include the vertex sets of dismantlable graphs (defined below), and vertex sets of single cycles (whose Vietoris--Rips complexes are homotopy equivalent to a wedge of spheres~\cite{Adamaszek2013}).
Examples of metric graphs whose Vietoris--Rips complexes are well-understood include trees, and single cycles (whose Vietoris--Rips complexes are typically homotopy equivalent to a single odd-dimensional sphere~\cite{AdamaszekAdams2017}).

Let $G$ be a graph with vertex set $V$ and with all edges of length one.\footnote{We make this assumption for simplicity's sake, even though it can be relaxed.} The vertex set $V$ is a metric space equipped with the shortest path metric.
We say that a vertex $v\in V$ is \emph{dominated} by $u\in V$ if $v$ is connected to $u$, and if each neighbor of $v$ is also a neighbor of $u$.
We say that a graph is \emph{dismantlable} if we can iteratively remove dominated vertices from $G$ in order to obtain the graph with a single vertex.
Note that if $v$ is dominated by $u$, then $v$ is dominated by $u$ in the 1-skeleton of $\vr{V}{r}$ for all $r\ge 1$.
It follows from the theory of folds, elementary simplicial collapses, or LC reductions~\cite{BabsonKozlov2006,Bjorner1995,Matouvsek2008} that if $G$ is dismantlable, then $\vr{V}{r}$ is contractible for all $r\ge 1$.
Examples of dismantlable graphs include trees, chordal graphs, and unit disk graphs of sufficiently dense samplings of convex sets in the plane~\cite[Lemma~2.1]{Latschev2001}.
We will also need the notion of a \emph{$k$-cycle} graph, a simple cycle with $k$ vertices and $k$ edges. 

The following proposition specifies a family of finite metric spaces that can be characterized using the results in this paper. 

\begin{proposition}\label{prop:iter_gluing_vertex}
Let $G$ be a finite graph, with each edge of length one, that can be obtained from a vertex by iteratively attaching (i) a dismantlable graph or (ii) a $k$-cycle graph along a vertex or along a single edge.
Let $V$ be the vertex set of the graph $G$.
Then we have $\vr{V}{r} \simeq \bigvee_{i=1}^n \vr{V(C_{k_i})}{r}$ for $r\ge 1$, where $n$ is the number of times operation (ii) is performed, $k_i$ are the corresponding cycle lengths, and $V(C_{k_i})$ is the vertex set of a $k_i$-cycle.
\end{proposition}

\begin{proof}
It suffices to show that an operation of type (i) does not change the homotopy type of the Vietoris--Rips complex of the vertex set, and that an operation of type (ii) has the effect up to homotopy of taking a wedge sum with $\vr{V(C_k)}{r}$.
The former follows from applying Theorem~\ref{thm:graph-gluing}, as the Vietoris--Rips complex of the vertex set of a dismantlable graph is contractible for all $r\ge 1$, and the latter also follows from Theorem~\ref{thm:graph-gluing}.
\end{proof}

The iterative procedure outlined in Proposition~\ref{prop:iter_gluing_vertex} can be used to obtain some recognizable families of graphs.
Examples include trees and wedge sums of cycles (Figure~\ref{fig:types}(a)).
More complicated are \emph{polygon trees}~\cite{DongKohTeo2005} in which cycles are iteratively attached along a single edge.
Graph (b) in Figure~\ref{fig:types} is an example that is built by using both (i) and (ii). 
Note that we can characterize graphs beyond what is stated in Proposition \ref{prop:iter_gluing_vertex}. 
Graph (c) in Figure~\ref{fig:types} can be characterized by iteratively gluing, for instance, the top two 7-cycle graphs along the common path of length 1, and then gluing the bottom 7-cycle graph along the path of length 2 (instead of length $1$ as required in Proposition \ref{prop:iter_gluing_vertex}).
Indeed, recall Theorem~\ref{thm:graph-gluing} can characterize graphs as long as all non-endpoint vertices have degree 2 with respect to one side of the gluing path, and that the gluing path is of length less than $\frac{\ell}{3}$.

A similar procedure is possible for metric graphs, except that we must replace arbitrary dismantlable graphs with the specific case of trees.\footnote{The case of $C_3$, a cyclic graph with three unit-length edges, is instructive.
Since $C_3$ is dismantlable we have that $\vr{V(C_3)}{r}$ is contractible for any $r\ge 1$.
But since the metric graph $C_3$ is isometric to a circle of circumference 3, it follows from~\cite{AdamaszekAdams2017} that $\vr{C_3}{r}$ is not contractible for $0<r<\frac{3}{2}$.} Note that any tree, $T$, is obtained by starting with a single vertex and iteratively taking the wedge sum with a single edge.
This implies that $\vr{T}{r}$ is contractible; thus, the persistent homology filtration of any tree is trivial, a result that is also established in~\cite{chan2013topology,OudotSolomon2017}.

\begin{proposition}\label{prop:iter_gluing_metric}
Let $G$ be a metric graph, with each edge of length one, that can be obtained from a vertex by iteratively attaching (i) an edge along a vertex or (ii) a $k$-cycle graph along a vertex or a single edge.
Then we have $\vr{G}{r} \simeq \bigvee_{i=1}^n \vr{C_{k_i}}{r}$ for $r\ge 1$, where $n$ is the number of times operation (ii) is performed, $k_i$ are the corresponding cycle lengths, and $C_{k_i}$ is a loop of length $k_i$.
\end{proposition}

\begin{proof}
The proof is analogous to that of Proposition~\ref{prop:iter_gluing_vertex}.
\end{proof}

The edge lengths can be generalized to allow any real-valued length, as long as the conditions of Theorem~\ref{thm:graph-gluing} are satisfied.
The iterative procedures in Propositions~\ref{prop:iter_gluing_vertex}--\ref{prop:iter_gluing_metric} produce families of metric graphs which can be characterized, but there are other ways to build up graphs in an admissible manner.
For example, instead of requiring that we glue along a single edge we may allow gluing along longer paths that meet the length criteria of Theorem~\ref{thm:graph-gluing}.
This more general definition allows for gluing along subpaths of previous gluing paths.
When using this procedure, we caution the reader that the order in which one glues graphs together matters.
In particular, one must glue first along the longest path before gluing along any shorter paths contained within.
A simple example is shown in Figure~\ref{fig:caution_gluing}.
Notice that we cannot characterize the vertex set of this graph iteratively using our results if we first glue $C_9$ to $C_3$ along a single edge.
Doing so would require that $C_{10}$ be glued along the path of length 3 next.
However, since $C_3$ is included, this path is no longer an admissible gluing path since $\alpha \nless \frac{\ell}{3}$.
This observation that gluing order matters does not come into play in Propositions~\ref{prop:iter_gluing_vertex}--\ref{prop:iter_gluing_metric} because gluing was restricted along a single edge.

\begin{figure}[!ht]
\begin{center}
\includegraphics[width=0.3\textwidth]{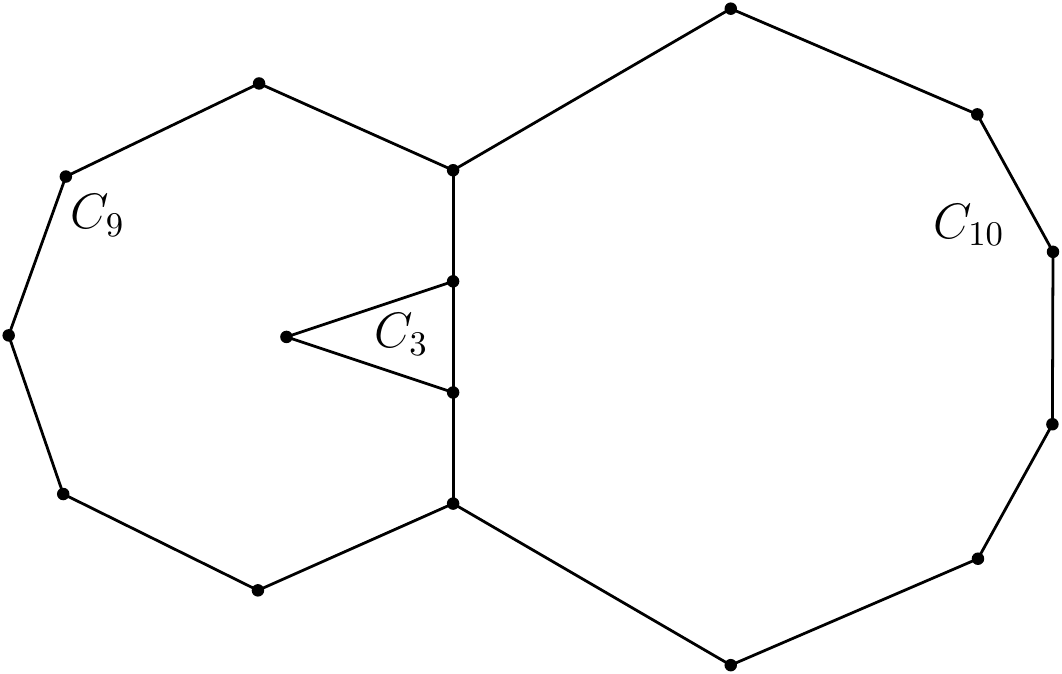}
\qquad
\includegraphics[width=0.3\textwidth]{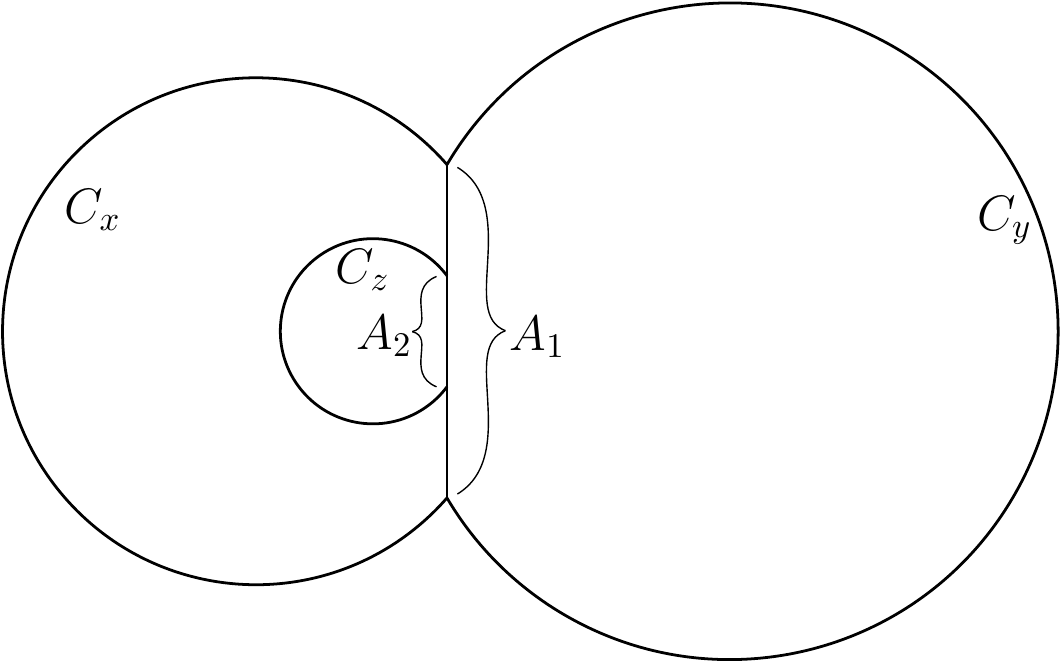}
\caption{A finite vertex subset (left) and a metric graph (right) built by first gluing $C_x$ to $C_y$ along $A_1$, and then gluing $C_z$ to the result along $A_2$.}\label{fig:caution_gluing}
\end{center}
\end{figure}

In future work, we hope to extend the results in this paper to gluing metric graphs along admissible isometric trees (a generalization of isometric simple paths).
The final graph (d) in Figure~\ref{fig:types}, the cube graph, is an example of a case for which Theorem~\ref{thm:graph-gluing} is not applicable.
We cannot compute the homology of the vertex set of the cube as the direct sum of the homology groups of smaller component pieces.
Indeed, if $V$ is the vertex set of the cube with each edge of length one, then $\dim(H_3(\vr{V}{2}))=1$ since $\vr{V}{2}$ is homotopy equivalent to the 3-sphere.
However, this graph is the union of five cycles of length four, and the Vietoris--Rips complex of the vertex set of a cycle of length four never has any 3-dimensional homology.


\section{Gluings with the supremum metric}\label{sec:sup}

Like the cube example in Section~\ref{sec:types}, we now describe other graphs that the techniques presented so far in this paper do not apply to.
Another such example is a ``circular ladder", i.e., the union of two circles with a finite number of edges or ``rungs" connecting corresponding angles on the two circles; see Figure~\ref{fig:ladders}(right).
Forming this ladder by gluing together smaller circles would require, in the final stages, gluing along paths that are too long for the results in our paper to apply.
In this section, we present a different set of techniques, based on Quillen's Fiber Theorem~A~\cite{barmak2011quillen,Quillen1973}, which one can use to describe the Vietoris--Rips complexes of a circular ladder (for example), \emph{albeit equipped with a distance that is not the metric graph distance}.

Specifically, in this section we consider the $L^\infty$ or \emph{supremum} metric on the product of two metric spaces.

\begin{definition}\label{def:sup-metric}
Given two metric spaces $(X,d_X)$ and $(Y,d_Y)$, the $L^\infty$ or \emph{supremum} metric on $X\times Y$ is defined by
\[ d((x,y),(x',y'))=\max\{d_X(x,x'),d_Y(y,y')\}.\]
\end{definition}

The main result in this section is Theorem~\ref{thm:sup-metric}, which establishes a relation between the Vietoris--Rips complexes of a space and a special type of gluing or product under this supremum metric. 

\begin{theorem}\label{thm:sup-metric}
Let $X$ and $Y$ be metric spaces, with subsets $X_0\subseteq X$ and $Y_0\subseteq Y$ such that
$Y_0\neq\emptyset$.
Furthermore, assume that $\vr{Y_0}{r}\simeq *$ and $\vr{Y}{r}\simeq *$.
Then $\vr{X\times Y_0 \cup X_0\times  Y}{r}\simeq\vr{X}{r}$, where $X\times Y_0 \cup X_0\times Y$ is equipped with the $L^\infty$ metric (as a subset of $X\times Y$).
\end{theorem}

\begin{proof}
We first prove the case where $X$ and $Y$ are finite.
Consider the map $\phi \colon L \to K$, where  $L=\vr{X\times Y_0 \cup X_0\times Y}{r}$, where $K=\vr{X}{r}$, and where $\phi$ is defined as the simplicial map sending a vertex $(x,y)\in X\times Y_0 \cup X_0\times Y\subseteq X\times Y$ to the vertex $x\in X$.
By~\cite[Theorem~4.2\footnote{We have switched the roles of $L$ and $K$ for notational convenience.}]{barmak2011quillen} (Quillen's Fiber Theorem~A), in order to show that $\phi$ is a homotopy equivalence, it suffices to show that for any closed simplex $\sigma\in K$, we have that its preimage $\phi^{-1}(\sigma)$ is contractible.
Note that $\phi^{-1}(\sigma)$ is the simplicial complex
\[\phi^{-1}(\sigma)=\{\tau\subseteq \sigma\times Y_0 \cup (\sigma\cap X_0)\times Y~|~\diam(\pi_Y(\tau))\le r\},\]
where $\pi_Y$ is the projection onto $Y$, and where the diameter is taken in $Y$.

We consider two cases.
If $\sigma\cap X_0=\emptyset$, then the definition of $\phi^{-1}(\sigma)$ simplifies down to the simplicial complex $\phi^{-1}(\sigma)=\sigma\times\vr{Y_0}{r}$, which is contractible as desired.
\newline\indent
For the second case, suppose $\sigma\cap X_0\neq\emptyset$.
The maximal simplices of $\phi^{-1}(\sigma)$ are of the form \[\kappa_\rho:=\sigma\times (\rho\cap Y_0)\cup(\sigma \cap X_0)\times\rho,\] where $\rho$ varies over the maximal simplices of $\vr{Y}{r}$.
For $Q$ is a simplicial complex, let $\nn(Q)$ denote the nerve of the cover of $Q$ by its maximal simplices.
Since this is a good cover, we have $\nn(Q)\simeq Q$ by the nerve lemma~\cite[Remark~15.22]{Kozlov2008}.
In particular, we have $\nn(\vr{Y}{r})\simeq \vr{Y}{r}$ and $\nn(\phi^{-1}(\sigma))\simeq \phi^{-1}(\sigma).$
It then suffices to show that maximal simplices $\rho,\rho'$ of $\vr{Y}{r}$ intersect if and only if the corresponding maximal simplices $\kappa_\rho,\kappa_{\rho'}$ of $\phi^{-1}(\sigma)$ intersect.
Indeed, that would give a simplicial isomorphism $\nn(\vr{Y}{r})\cong \nn(\phi^{-1}(\sigma))$, in which case we would have 
\[ \phi^{-1}(\sigma)\simeq \nn(\phi^{-1}(\sigma))\cong \nn(\vr{Y}{r})\simeq \vr{Y}{r}\simeq*,\]
as desired.
If $\kappa_\rho$ and $\kappa_{\rho'}$ intersect, then they do so at a point in both $\sigma\times\rho$ and $\sigma\times\rho'$, which means that $\rho$ and $\rho'$ intersect.
For the reverse direction, suppose $\rho$ and $\rho'$ intersect.
Since $\sigma\cap X_0\neq\emptyset$, then clearly $\kappa_\rho$ and $\kappa_{\rho'}$ also intersect, as needed.
Hence we have shown that $\phi^{-1}(\sigma)$ is contractible, and so~\cite[Theorem~4.2]{barmak2011quillen} implies that $\phi\colon L\to K$ is a homotopy equivalence.

We now use the finite case to prove the arbitrary case (where $X$ and $Y$ are possibly infinite).
Our proof relies on Lemma~\ref{lem:compactness-Quillen}.
Consider $\phi\colon L\to K$ defined as above, where again $L=\vr{X\times Y_0 \cup X_0\times Y}{r}$ and $K=\vr{X}{r}$.

In order to apply Lemma~\ref{lem:compactness-Quillen}, we have to show that for every finite $V_0\subseteq X\times Y_0 \cup X_0\times Y$ and $W_0\subseteq X$, there exists a finite subset $V_1$ with $V_0\subseteq V_1\subseteq X\times Y_0 \cup X_0\times Y$ and $W_0\subseteq \phi(V_1)\subseteq X$ such that the induced map $L[V_1]\to K[\phi(V_1)]$, i.e.\ $\vr{V_1}{r}\to\vr{\phi(V_1)}{r}$, is a homotopy equivalence.
Let $V_0^X$ and $V_0^Y$ be the images of the projections of $V_0$ onto $X$ and onto $Y$, respectively.
If $V_0^Y\cap Y_0=\emptyset$ then, for the remainder of this proof add to $V_0^Y$ a single arbitrary point of $Y_0$, enforcing $V_0^Y\cap Y_0\neq\emptyset$.
We let
\[V_1=\left((V_0^X\cup W_0)\times (V_0^Y\cap Y_0)\right) \cup \left(\left((V_0^X\cup W_0)\cap X_0\right) \times V_0^Y\right).\]
To see that $V_0\subseteq V_1$, note that any point $(x,y)\in V_0$ has either $x\in X_0$ or $y\in Y_0$ (or both).
Note that $V_1$ is finite, and that $\phi(V_1)=V_0^X\cup W_0$ contains $W_0$ since $V_0^Y\cap Y_0\neq\emptyset$.
We then apply Theorem~\ref{thm:sup-metric} in the finite case 
to get that the induced map $\vr{V_1}{r}\to\vr{\phi(V_1)}{r}$ is a homotopy equivalence, as required.
Hence, Lemma~\ref{lem:compactness-Quillen} implies that $\phi$ is a homotopy equivalence.
\end{proof}

\begin{remark}
We remark that Theorem~\ref{thm:sup-metric} is not a consequence of~\cite[Proposition~10.2]{AdamaszekAdams2017}, which considers only the Vietoris--Rips complexes of $X\times Y$, but not the Vietoris--Rips complexes of any subsets thereof.
\end{remark}

We give an example that relies on Theorem~\ref{thm:sup-metric}.

\begin{example}\label{ex:circ-ladder1}
Let $Y=[0,\ell]$ be an interval, and let $\emptyset\neq Y_0\subseteq Y$ consist of $m$ points $0\le y_1<y_2<\ldots< y_m\le \ell$.
Let $\kappa_Y=\max_{i\in\{1,\ldots,m-1\}}(y_{i+1}-y_i)$.
Then for $X$ an arbitrary metric space and $X_0$ an arbitrary subset thereof, and for $r\ge \kappa_Y$ (which implies that $\vr{Y_0}{r}\simeq *$), Theorem~\ref{thm:sup-metric} implies that $\vr{X\times Y_0 \cup X_0\times Y}{r}\simeq\vr{X}{r}$, where $X\times Y_0 \cup X_0\times Y$ is equipped with the $L^\infty$ metric.
For example, by taking $X$ to be a circle, by taking $X_0$ to be an arbitrary subset of ``rung locations", and by taking $Y_0=\{0=y_1, y_2, y_3, y_4=\ell\}\subseteq Y$, then $X\times Y_0 \cup X_0\times Y$ is the circular ladder on the left in Figure~\ref{fig:ladders}.
So the Vietoris--Rips complex of this circular ladder is homotopy equivalent to the Vietoris--Rips complex of the circle $X$ (see~\cite{AdamaszekAdams2017}) for $r\ge \kappa_Y$.
We can increase the ``width" of this circular ladder, and have the same result apply, simply by increasing the number of points in $Y_0$.
\end{example}

\begin{figure}[!h]
\begin{center}
\includegraphics[width=0.7\linewidth]{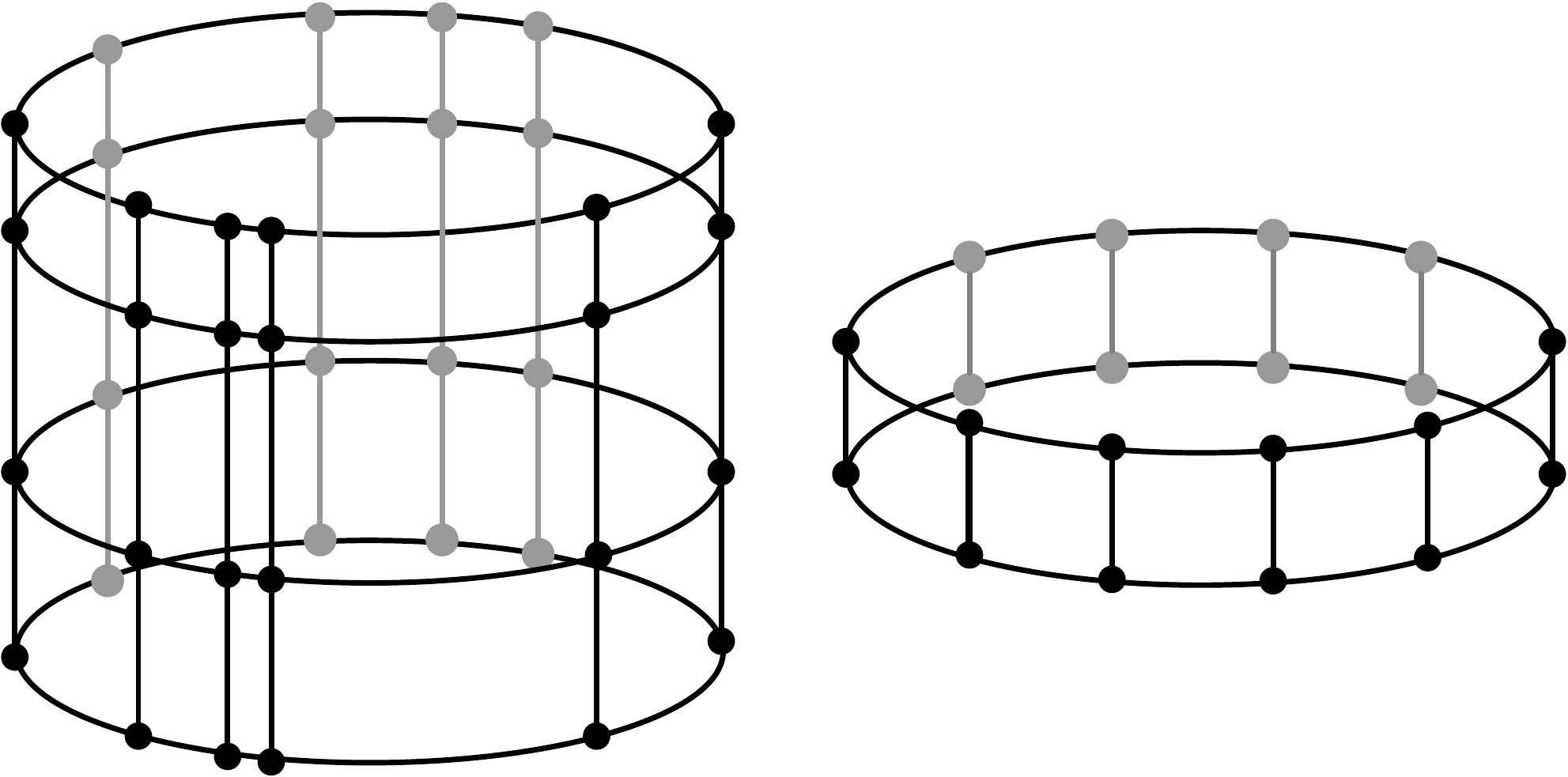}
\caption{Circular ladders with unequally (left) and equally (right) spaced rings, as in Examples~\ref{ex:circ-ladder1} and~\ref{ex:circ-ladder2}, respectively.}
\label{fig:ladders}
\end{center}
\end{figure}

\begin{question}
What are the homotopy types of the Vietoris--Rips complexes of a circular ladder, when the circular ladder is instead equipped with the metric graph distance?
\end{question}

When the circular ladder is of ``width one" and as symmetric as possible, the below example shows that our Theorem~\ref{thm:graph-gluing} can help give a partial answer to the above question.

\begin{example}\label{ex:circ-ladder2}
Let $Y=[0,1]$ be the unit interval, and let $Y_0=\{0,1\}\subseteq Y$ be the two endpoints thereof.
Let $X=S^1$ be a circle of integral circumference $n\ge 3$.
Let $X_0\subseteq X$ consist of $n$ evenly spaced points along the circle, i.e., we will have $n$ ``rungs" in our circular ladder.
See Figure~\ref{fig:ladders} on the right.
We equip the circular ladder $X\times Y_0 \cup X_0\times Y$ with the metric graph distance.

For $r < \frac{n}{3}$, we can describe the homotopy type of $\vr{X\times Y_0 \cup X_0\times Y}{r}$ as follows.
Let $Z=\mathbb{R}\times Y_0 \cup \Z\times Y$ be the ``infinite" ladder.
Let $G$ be the group of integers under addition.
Let $G$ act on $Z$ by having integer 1 send $(x,y)$ to $(x+n,y)$, and hence, integer $j$ sends $(x,y)$ to $(x+jn,y)$.
It turns out that $G$ also acts on $\vr{Z}{r}$ in a similar way.
Note that the quotient space $Z/G$ under the group action is the circular ladder, namely $Z/G=X\times Y_0 \cup X_0\times Y$.
For $r < \frac{n}{3}$, it follows from~\cite{AdamsHeimPeterson} 
that $\vr{Z/G}{r}$, the Vietoris--Rips complex of the circular ladder, and $\vr{Z}{r}/G$, the quotient  of the Vietoris--Rips complex of the infinite ladder under the action of $G$, are isomorphic as simplicial complexes.
Note that the infinite ladder $Z$ is formed by gluing together coutably many circles of circumference 4 along sufficiently short gluing paths.
By~\cite{AdamaszekAdams2017}, the Vietoris--Rips complex of a circle of circumference 4 is homotopy equivalent to the $(2k+1)$-dimensional sphere for $\frac{4k}{2k+1}<r<\frac{4(k+1)}{2k+3}$ with $k$ an integer.
Since $Z$ is formed by gluing these circles together along sufficiently short gluing paths, by Theorem~\ref{thm:graph-gluing} we know that for complex $\vr{Z}{r}\simeq \bigvee^\infty S^{2k+1}$ is homotopy equivalent to a countably infinite wedge sum of $(2k+1)$-dimensional spheres for $\frac{4k}{2k+1}<r<\frac{4(k+1)}{2k+3}$.
Hence when $r < \frac{n}{3}$, our understanding of $G$'s action on $\vr{Z}{r}$ implies that if furthermore $\frac{4k}{2k+1}<r<\frac{4(k+1)}{2k+3}$, then   $\vr{X\times Y_0 \cup X_0\times Y}{r}=\vr{Z/G}{r}\simeq S^1\vee(\bigvee^n S^{2k+1})$ is homotopy equivalent to the wedge sum of a single circle $S^1$ together with the $n$-fold wedge sum of spheres of dimension $2k+1$.
\end{example}


\section{Discussion}
\label{sec:discussion}

We have shown that the wedge sum of Vietoris--Rips complexes is homotopy equivalent to the corresponding complex for the metric wedge sum, and generalized this result in the case of Vietoris--Rips complexes for certain metric space gluings.
Our ultimate goal is to understand to the greatest extent possible the topological structure of large classes of metric graphs via persistent homology.
Building on previous work in~\cite{AdamaszekAdams2017} and~\cite{GasparovicGommelPurvine2018},
the results in this paper constitute another important step toward this goal by providing a characterization of the persistence profiles of metric graphs obtainable via certain types of metric gluing.
Many interesting questions remain for future research.

\para{Gluing beyond a single path.} We are interested in studying metric graphs obtainable via metric gluings other than along single paths, such as gluing along a tree, or along subgraphs with non-trivial topology (e.g., multiple components, or containing a cycle). 
Furthermore, the techniques of our paper do not allow one to analyze self-gluings such as forming an $n$-cycle $C_n$ from a path of length $n$.
Self-gluings also may change the metric structure significantly, and it is likely that new and very different techniques need to be developed to handle self-gluings.

\para{Generative models for metric graphs.}
Our results can be considered as providing a generative model for metric graphs, where we specify a particular metric gluing rule for which we have a clear understanding of its effects on persistent homology.
Expanding the list of metric gluing rules would in turn lead to a larger collection of generative models.

\para{Approximations of persistent homology profiles.}
A particular metric graph that arises from data in practice may not directly correspond to an existing generative model.
However, we may still be able to approximate its persistent homology profile via stability results (e.g.~\cite{ChazalCohen-SteinerGuibas2009,Turner2016}) by demonstrating close proximity between its metric and a known one.

\section*{Acknowledgments} 
We are grateful for the Women in Computational Topology (WinCompTop) workshop for initiating our research collaboration.
In particular, participant travel support was made possible through the grant NSF-DMS-1619908.
This collaborative group was also supported by the American Institute of Mathematics (AIM) Structured Quartet Research Ensembles (SQuaRE) program.
EP was supported by the High Performance Data Analytics (HPDA) program at Pacific Northwest National Laboratory. 
RS is partially supported by the SIMONS Collaboration grant 318086 and NSF grant DMS-1854705. 
BW is partially supported by the NSF grants DBI-1661375 and IIS-1513616. 
YW is partially supported by National Science Foundation (NSF) via grants CCF-1526513, CCF-1618247, CCF-1740761, and DMS-1547357.
LZ is partially supported by the NSF grant CDS\&E-MSS-1854703.


\begin{thebibliography}{10}

\bibitem{AanChaChe2012}
M.~Aanjaneya, F.~Chazal, D.~Chen, M.~Glisse, L.~Guibas, and D.~Morozov.
\newblock Metric graph reconstruction from noisy data.
\newblock {\em International Journal of Computational Geometry and
  Applications}, 22(04):305--325, 2012.

\bibitem{Adamaszek2013}
M.~Adamaszek.
\newblock Clique complexes and graph powers.
\newblock {\em Israel Journal of Mathematics}, 196(1):295--319, 2013.

\bibitem{AdamaszekAdams2017}
M.~Adamaszek and H.~Adams.
\newblock The {V}ietoris--{R}ips complexes of a circle.
\newblock {\em Pacific Journal of Mathematics}, 290:1--40, 2017.

\bibitem{AdamaszekAdamsFrick2016}
M.~Adamaszek, H.~Adams, F.~Frick, C.~Peterson, and C.~Previte-Johnson.
\newblock Nerve complexes of circular arcs.
\newblock {\em Discrete {\&} Computational Geometry}, 56(2):251--273, 2016.

\bibitem{AdamaszekAdamsGasparovic2018}
M.~Adamaszek, H.~Adams, E.~Gasparovic, M.~Gommel, E.~Purvine, R.~Sazdanovic,
  B.~Wang, Y.~Wang, and L.~Ziegelmeier.
\newblock {V}ietoris-{R}ips and {C}ech complexes of metric gluings.
\newblock In B.~Speckmann and C.~D. T\'{o}th, editors, {\em 34th International
  Symposium on Computational Geometry}, volume~99 of {\em Leibniz International
  Proceedings in Informatics (LIPIcs)}, pages 3:1--3:15, Dagstuhl, Germany,
  2018. Schloss Dagstuhl--Leibniz-Zentrum fuer Informatik.

\bibitem{AdamsHeimPeterson}
H.~Adams, M.~Heim, and C.~Peterson.
\newblock Metric thickenings and group actions.
\newblock In preparation, 2019.

\bibitem{BabsonKozlov2006}
E.~Babson and D.~N. Kozlov.
\newblock Complexes of graph homomorphisms.
\newblock {\em Israel Journal of Mathematics}, 152(1):285--312, 2006.

\bibitem{barmak2011quillen}
J.~A. Barmak.
\newblock On {Q}uillen's {T}heorem {A} for posets.
\newblock {\em Journal of Combinatorial Theory, Series A}, 118(8):2445--2453,
  2011.

\bibitem{BarmakMinian2008}
J.~A. Barmak and E.~G. Minian.
\newblock Simple homotopy types and finite spaces.
\newblock {\em Advances in Mathematics}, 218:87--104, 2008.

\bibitem{BiswalYetkinHaughton1995}
B.~Biswal, F.~Z. Yetkin, V.~M. Haughton, and J.~S. Hyde.
\newblock Functional connectivity in the motor cortex of resting human brain
  using echo-planar {MRI}.
\newblock {\em Magnetic Resonance in Medicine}, 34:537--541, 1995.

\bibitem{Bjorner1995}
A.~Bj{\"o}rner.
\newblock Topological methods.
\newblock {\em Handbook of Combinatorics}, 2:1819--1872, 1995.

\bibitem{BridsonHaefliger1999}
M.~R. Bridson and A.~Haefliger.
\newblock {\em Metric Spaces of Non-Positive Curvature}.
\newblock Springer-Verlag, 1999.

\bibitem{BuragoBuragoIvanov2001}
D.~Burago, Y.~Burago, and S.~Ivanov.
\newblock {\em A course in metric geometry}, volume~33.
\newblock American Mathematical Society Providence, RI, 2001.

\bibitem{Carlsson2009}
G.~Carlsson.
\newblock Topology and data.
\newblock {\em Bulletin of the American Mathematical Society}, 46(2):255--308,
  2009.

\bibitem{Chacholski}
W.~Chach\'{o}lski, A.~Jin, M.~Scolamiero, and F.~Tombari.
\newblock Homotopical decompositions of {V}ietoris--{R}ips complexes.
\newblock Forthcoming, 2019.

\bibitem{chan2013topology}
J.~M. Chan, G.~Carlsson, and R.~Rabadan.
\newblock Topology of viral evolution.
\newblock {\em Proceedings of the National Academy of Sciences},
  110(46):18566--18571, 2013.

\bibitem{ChazalCohen-SteinerGuibas2009}
F.~Chazal, D.~Cohen-Steiner, L.~J. Guibas, F.~M\'{e}moli, and S.~Y. Oudot.
\newblock Gromov-{H}ausdorff stable signatures for shapes using persistence.
\newblock {\em Computer Graphics Forum}, 28(5):1393--1403, 2009.

\bibitem{ChazalDe-SilvaGlisse2016}
F.~Chazal, V.~de~Silva, M.~Glisse, and S.~Oudot.
\newblock {\em The structure and stability of persistence modules}.
\newblock Springer, 2016.

\bibitem{ChazalSilvaOudot2013}
F.~Chazal, V.~de~Silva, and S.~Oudot.
\newblock Persistence stability for geometric complexes.
\newblock {\em Geometriae Dedicata}, pages 1--22, 2013.

\bibitem{DongKohTeo2005}
F.~Dong, K.-M. Koh, and K.~L. Teo.
\newblock {\em Chromatic Polynomials and Chromaticity of Graphs}.
\newblock World Scientific, 2005.

\bibitem{EdelsbrunnerHarer2010}
H.~Edelsbrunner and J.~L. Harer.
\newblock {\em Computational Topology: An Introduction}.
\newblock American Mathematical Society, 2010.

\bibitem{GasparovicGommelPurvine2018}
E.~Gasparovic, M.~Gommel, E.~Purvine, R.~Sazdanovic, B.~Wang, Y.~Wang, and
  L.~Ziegelmeier.
\newblock A complete characterization of the one-dimensional intrinsic \v{C}ech
  persistence diagrams for metric graphs.
\newblock In {\em Research in Computational Topology}, 2018.

\bibitem{Hatcher2002}
A.~Hatcher.
\newblock {\em Algebraic Topology}.
\newblock Cambridge University Press, 2002.

\bibitem{Hausmann1995}
J.-C. Hausmann.
\newblock On the {V}ietoris--{R}ips complexes and a cohomology theory for
  metric spaces.
\newblock {\em Annals of Mathematics Studies}, 138:175--188, 1995.

\bibitem{Kozlov2008}
D.~N. Kozlov.
\newblock {\em Combinatorial Algebraic Topology}, volume~21 of {\em Algorithms
  and Computation in Mathematics}.
\newblock Springer-Verlag, 2008.

\bibitem{Kuchment2004}
P.~Kuchment.
\newblock Quantum graphs {I}. {S}ome basic structures.
\newblock {\em Waves in Random Media}, 14(1):S107--S128, 2004.

\bibitem{Latschev2001}
J.~Latschev.
\newblock Vietoris--{R}ips complexes of metric spaces near a closed
  {R}iemannian manifold.
\newblock {\em Archiv der Mathematik}, 77(6):522--528, 2001.

\bibitem{LesnickRabadanRosenbloom}
M.~Lesnick, R.~Rabadan, and D.~Rosenbloom.
\newblock Quantifying genetic innovation: Mathematical foundations for the
  topological study of reticulate evolution.
\newblock In preparation, 2018.

\bibitem{Matouvsek2008}
J.~Matou{\v{s}}ek.
\newblock {LC} reductions yield isomorphic simplicial complexes.
\newblock {\em Contributions to Discrete Mathematics}, 3(2), 2008.

\bibitem{OudotSolomon2017}
S.~Oudot and E.~Solomon.
\newblock Barcode embeddings, persistence distortion, and inverse problems for
  metric graphs.
\newblock {\em arXiv: 1712.03630}, 2017.

\bibitem{Quillen1973}
D.~Quillen.
\newblock Higher algebraic k-theory: I.
\newblock In H.~Bass, editor, {\em Higher K-Theories}, pages 85--147, Berlin,
  Heidelberg, 1973. Springer Berlin Heidelberg.

\bibitem{SousbiePichonKawahara2011}
T.~Sousbie, C.~Pichon, and H.~Kawahara.
\newblock The persistent cosmic web and its filamentary structure -- {II.
  I}llustrations.
\newblock {\em Monthly Notices of the Royal Astronomical Society},
  414(1):384--403, 2011.

\bibitem{Turner2016}
K.~Turner.
\newblock Generalizations of the {R}ips filtration for quasi-metric spaces with
  persistent homology stability results.
\newblock {\em ar{X}iv:1608.00365}, 2016.

\bibitem{Virk2017}
{\v{Z}}.~Virk.
\newblock 1-dimensional intrinsic persistence of geodesic spaces.
\newblock {\em arXiv:1709.05164}, 2017.

\end{thebibliography}


\appendix

\section{Counterexample to~\cite[Corollary~9]{AdamaszekAdamsGasparovic2018}}\label{appendix:counterexample}

We thank Wojciech Chach\'{o}lski, Alvin Jin, Martina Scolamiero, and Francesca Tombari for the following counterexample to~\cite[Corollary~9]{AdamaszekAdamsGasparovic2018} in the conference version of this paper.
We would also like to mention their forthcoming work on homotopical decompositions of Vietoris--Rips complexes~\cite{Chacholski}.

As shown in Figure \ref{fig:counterexample}, suppose $X$ consists of 4 points in the shape of a quadrilateral of side lengths 0.5, 0.6, 0.5, and 0.6, and with diagonals of length 1.1.
Also, let $Y$ consist of 3 points in the shape of a triangle of side lengths 0.5, 0.5, and 0.6.
Their intersection $A=X\cap Y$ is two points at a distance of 0.6.
For $r=1$ the gluing of the corresponding Vietoris--Rips complexes is homotopy equivalent to a circle since $\vr{X}{r} \simeq S^1$ and $\vr{Y}{r}$ is contractible.
However, the Vietoris--Rips complex of the gluing is in fact contractible (a cone with apex the single point in $Y\setminus A$), and hence not homotopy equivalent to the gluing of the Vietoris--Rips complexes.

\begin{figure}[htb]
\begin{center}
\includegraphics[width=0.5\linewidth]{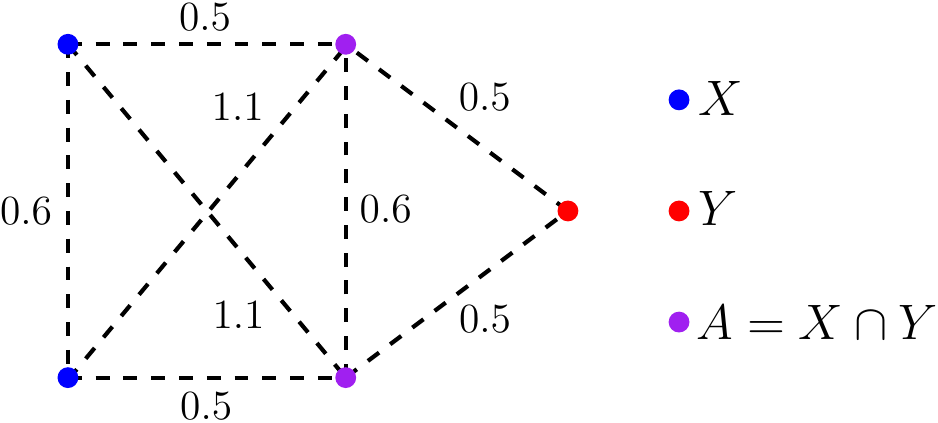}
\caption{A counterexample to~\cite[Corollary~9]{AdamaszekAdamsGasparovic2018}, where metric spaces $X$, $Y$, and $A = X \cap Y$ are denoted in blue, red, and purple, respectively.
Distances between each pair of points are indicated on the dashed lines.}
\label{fig:counterexample}
\end{center}
\end{figure}

\end{document}